\documentclass[a4paper,11pt]{article}
\usepackage{amssymb,amsmath,amsthm,amsfonts}
\usepackage{bookmark}
\usepackage{mathrsfs}
\usepackage{multirow}
\usepackage{graphicx}
\usepackage{authblk}
\usepackage{indentfirst}
\usepackage{multicol}
\usepackage{tabu}
\usepackage{url}
\usepackage{fancyhdr}
\usepackage[numbers]{natbib}
\usepackage[all]{xy}
\usepackage{mathtools}
\usepackage[english]{babel}
\usepackage{colortbl}
\usepackage{caption}
\usepackage{subcaption}
\usepackage{tikz}\usepackage{float}\usepackage{wrapfig}
\usepackage{hyperref}
\usepackage{tikz}
\usepackage{enumitem}

\tikzset{overcross/.style={double, line width=1.5, white, double=#1, double distance=\knotlinewidth},
    overcross/.default={black},
    knot/.style={line width=\knotlinewidth, baseline=-.5ex}}
\newcommand{\knotlinewidth}{1pt}
\newcommand{\RIa}[1][]{\tikz[knot, #1]{\draw(-.5,.5) to[out=-90,in=-90] (.5,0); \draw[overcross] (.5,0) to[out=90,in=90] (-.5,-.5);}}
\newcommand{\RIb}[1][]{\tikz[knot, #1]{\draw[looseness=.8] (-.5,-.5) to[out=90, in=-90] (.5,0) to[out=90, in=-90] (-.5,.5);}}
\newcommand{\RIIa}[1][]{\tikz[knot, #1]{\draw[looseness=2.3] (-.5,-.5) to[out=0, in=0] (-.5,.5); \draw[looseness=2.3, overcross=black] (.5,-.5) to[out=180, in=180] (.5,.5);}}
\newcommand{\RIIb}[1][]{\tikz[knot, #1]{\draw[looseness=1.4] (-.5,-.5) to[out=0, in=0] (-.5,.5);\draw[looseness=1.4] (.5,.5) to[out=180,in=180] (.5,-.5);}}
\newcommand{\RIIIa}[1][]{\tikz[knot, #1]{\draw[black] (-120:.58) to[out=60,in=-120] (150:.2) to[out=60, in=-120] (60:.58); \draw[rotate=-120, overcross=black] (-120:.58) to[out=60, in=-120] (150:.2) to[out=60, in=-120] (60:.58); \draw[rotate=120, overcross=black!] (-120:.58) to[out=60, in=-120] (150:.2) to[out=60, in=-120] (60:.58);}}

\newcommand{\Tan}[1][1]{\tikz[knot, scale=#1]{\draw[looseness=2.3] (0,1) to[out=-30, in=-30] (0,0.3); \draw[looseness=1.3,overcross=black] (0,0) to[out=70, in=210] (0.7,1);\draw[looseness=2.3,overcross=black] (0,0.3) to[out=150, in=100] (0.6,0);\draw[looseness=0] (-0.4,0) to (0.9,0);\draw[looseness=0] (-0.4,1) to (0.9,1);}}

\newcommand{\Tano}[1][1]{\tikz[knot, scale=#1]{\clip (0.3,0.5) circle(0.515);\draw[looseness=2.3] (0,1) to[out=-30, in=-30] (0,0.3); \draw[looseness=1.3,overcross=black] (0,0) to[out=70, in=210] (0.7,1);\draw[looseness=2.3,overcross=black] (0,0.3) to[out=150, in=100] (0.6,0);\draw(0.3,0.5) circle(0.5);}}

\newcommand{\Tangle}[1][1]{\tikz[knot, scale=#1]{\draw[looseness=0.8,overcross=black] (-0.4,0.3) to[out=-70, in=180] (0,-0.1);\draw[looseness=1.8,overcross=black] (0.4,-0.3) to[out=110, in=70](-0.4,-0.3);
\draw[looseness=0.8,overcross=black] (0.4,0.3) to[out=-110, in=0] (0,-0.1);\draw[looseness=0.8,overcross=black] (0.4,0.3) to[out=-110, in=0] (0,-0.1);\draw(0,0) circle(0.5);}}

\newcommand{\Tanglo}[1][1]{\tikz[knot, scale=#1]{\draw[looseness=0.8,overcross=black] (-0.4,0.3) to[out=-70, in=180] (0,-0.1);\draw[looseness=1.8,overcross=black] (0.4,-0.3) to[out=110, in=70](-0.4,-0.3);\draw[looseness=0.8,overcross=black] (0.4,0.3) to[out=-110, in=0] (0,-0.1);\draw[looseness=0.8,overcross=black] (0.4,0.3) to[out=-110, in=0] (0,-0.1);\draw(0,0) circle(0.5);\fill[white] (0.15,0) circle(0.15);\draw(0.15,0) circle(0.15);}}

\newcommand{\Tanglx}[1][1]{\tikz[knot, scale=#1]{\draw[looseness=0.8] (-0.4,0.3) to[out=-10, in=110] (0.4,-0.3);\draw[looseness=0.8,overcross=black] (0.4,0.3) to[out=-110, in=10] (-0.4,-0.3);\draw(0,0) circle(0.5);}}

\newtheorem{theorem}{Theorem}[section]
\newtheorem{proposition}[theorem]{Proposition}
\newtheorem{lemma}[theorem]{Lemma}

\newtheorem{definition}{Definition}[section]
\newtheorem{example}{Example}[section]
\newtheorem{remark}{Remark}[section]
\newcommand{\im}{{\mathrm{im}\hspace{0.1em}}}

\newcommand{\sgn}{{\mathrm{sgn}\hspace{0.1em}}}
\newcommand{\qdim}{{\mathrm{qdim}\hspace{0.1em}}}
\usepackage{xr}
\makeatletter
    \newcommand*{\addFileDependency}[1]{
    \typeout{(#1)}
    \@addtofilelist{#1}
    \IfFileExists{#1}{}{\typeout{No file #1.}}
    }
\makeatother

\newcommand{\rightcross}[1]{
  \begin{tikzpicture}[knot/.style={black}]
    \begin{scope}[xshift=5cm]
      \draw[line width=0.8pt] (-#1,#1)-- (#1,-#1);
      \draw[white,double=black,double distance=0.8pt,thick](-#1,-#1)-- (#1,#1);
      \draw[line width=0.4pt](0,0)circle (#1*1.414);
    \end{scope}
  \end{tikzpicture}
}

\newcommand{\arightcross}[1]{
  \begin{tikzpicture}[knot/.style={black}]
    \begin{scope}[xshift=5cm]
      \draw[line width=0.8pt,->]  (#1,-#1)--(-#1,#1);
      \draw[white,double=black,double distance=0.8pt,thick](-#1,-#1)-- (#1,#1);
      \draw[line width=0.8pt,->] (#1*0.5,#1*0.5)--(#1,#1);
      \draw[line width=0.4pt](0,0)circle (#1*1.414);
    \end{scope}
  \end{tikzpicture}
}

\newcommand{\leftcross}[1]{
  \begin{tikzpicture}[knot/.style={black}]
    \begin{scope}[xshift=5cm]
      \draw[line width=0.8pt] (-#1,-#1)-- (#1,#1);
      \draw[white,double=black,double distance=0.8pt,thick](-#1,#1)-- (#1,-#1);
      \draw[line width=0.4pt](0,0)circle (#1*1.414);
    \end{scope}
  \end{tikzpicture}
}

\newcommand{\aleftcross}[1]{
  \begin{tikzpicture}[knot/.style={black}]
    \begin{scope}[xshift=5cm]
      \draw[line width=0.8pt,->] (-#1,-#1)-- (#1,#1);
      \draw[white,double=black,double distance=0.8pt,thick](-#1,#1)-- (#1,-#1);
      \draw[line width=0.8pt,->] (-#1*0.5,#1*0.5)-- (-#1,#1);
      \draw[line width=0.4pt](0,0)circle (#1*1.414);
    \end{scope}
  \end{tikzpicture}
}

\newcommand{\udarc}[1]{
  \begin{tikzpicture}
    \draw[line width=0.8pt] (0,0) arc (140:40:#1cm);
    \draw[line width=0.8pt] (0,#1*1.2) arc (220:320:#1cm);
     \draw[line width=0.4pt](#1*0.766,#1*0.6)circle (#1);
  \end{tikzpicture}
}

\newcommand{\cudarc}[1]{
  \begin{tikzpicture}
    \draw[line width=0.8pt] (0,0) arc (140:40:#1cm);
    \draw[line width=0.8pt] (0,#1*1.2) arc (220:320:#1cm);
    \draw[line width=0.4pt] (-#1*0.333, -#1*0.5) -- (-#1*0.333 , #1*1.7)--(#1*1.866 , #1*1.7)--(#1*1.866 , -#1*0.5)--(-#1*0.333, -#1*0.5);
  \end{tikzpicture}
}

\newcommand{\cudarcc}[1]{
  \begin{tikzpicture}
    \draw[line width=0.8pt] (0,0) arc (140:40:#1cm);
    \draw[line width=0.8pt] (0,#1*1.2) arc (220:320:#1cm);
    \draw[line width=0.4pt] (-#1*0.233, -#1*0.5) -- (-#1*0.233 , #1*1.7)--(#1*3.5 , #1*1.7)--(#1*3.5 , -#1*0.5)--(-#1*0.233, -#1*0.5);
    \draw[line width=0.8pt] (#1*1.532,0) arc[start angle=-140, end angle=140, radius=#1*0.933];
  \end{tikzpicture}
}

\newcommand{\lcrossingarc}[1]{
  \begin{tikzpicture}
    \draw[line width=0.8pt] (#1,-#1) arc (-135:135:#1 *1.414cm);
    \draw[line width=0.8pt] (-#1*1.2,-#1*1.2)-- (#1,#1);
    \draw[white,double=black,double distance=0.8pt,thick](-#1*1.2,#1*1.2)-- (#1,-#1);
    \draw[line width=0.8pt](#1 *1.2,0)circle(#1 *2.68);
  \end{tikzpicture}
}

\newcommand{\rcrossingarc}[1]{
  \begin{tikzpicture}
    \draw[line width=0.8pt] (#1,-#1) arc (-135:135:#1 *1.414cm);
    \draw[line width=0.8pt](-#1*1.2,#1*1.2)-- (#1,-#1);
    \draw[white,double=black,double distance=0.8pt,thick] (-#1*1.2,-#1*1.2)-- (#1,#1);
    \draw[line width=0.8pt](#1 *1.2,0)circle(#1 *2.68);
  \end{tikzpicture}
}

\newcommand{\caarc}[1]{
  \begin{tikzpicture}
    \draw[line width=0.8pt] (0,0) arc (40:140:#1cm);
    \draw[line width=0.8pt] (0,#1*1.2) arc (320:220:#1cm);
    \draw[line width=0.4pt] (#1*0.233, -#1*0.5) -- (#1*0.233 , #1*1.7)--(-#1*4.727 , #1*1.7)--(-#1*4.727 , -#1*0.5)--(#1*0.233, -#1*0.5);
    \draw[line width=0.8pt] (-#1*1.532,0) arc[start angle=320, end angle=40, radius=#1*0.933];
    \draw[line width=0.8pt] (-#1*3.994,0) arc[start angle=-40, end angle=40, radius=#1*0.933];
  \end{tikzpicture}
}

\newcommand{\cudarco}[1]{
  \begin{tikzpicture}
    \draw[line width=0.4pt] (-#1*0.233, -#1*0.5) -- (-#1*0.233 , #1*1.7)--(#1*3.5 , #1*1.7)--(#1*3.5 , -#1*0.5)--(-#1*0.233, -#1*0.5);
    \draw[line width=0.8pt] (#1* 2.247,#1* 0.6)circle (#1*0.933);
    \draw[line width=0.8pt] (#1*0.5,0) arc[start angle=-40, end angle=40, radius=#1*0.933];
  \end{tikzpicture}
}

\newcommand{\caoa}[1]{
  \begin{tikzpicture}
    \draw[line width=0.4pt] (-#1*0.233, -#1*0.5) -- (-#1*0.233 , #1*1.7)--(#1*4.727 , #1*1.7)--(#1*4.727 , -#1*0.5)--(-#1*0.233, -#1*0.5);
    \draw[line width=0.8pt] (#1* 2.247,#1* 0.6)circle (#1*0.933);
    \draw[line width=0.8pt] (#1*0.5,0) arc[start angle=-40, end angle=40, radius=#1*0.933];
    \draw[line width=0.8pt] (#1*3.994,0) arc[start angle=220, end angle=140, radius=#1*0.933];
  \end{tikzpicture}
}

\newcommand{\ccircle}[1]{
  \begin{tikzpicture}
    \draw[line width=0.4pt] (-#1*1.233, -#1*0.5) -- (-#1*1.233 , #1*1.7)--(#1*3.5 , #1*1.7)--(#1*3.5 , -#1*0.5)--(-#1*1.233, -#1*0.5);
    \draw[line width=0.8pt] (#1* 2.247,#1* 0.6)circle (#1*0.933);
    \draw[line width=0.8pt] (#1* 0,#1* 0.6)circle (#1*0.933);
  \end{tikzpicture}
}

\newcommand{\cooo}[1]{
  \begin{tikzpicture}
    \draw[line width=0.4pt] (-#1*1.233, -#1*0.5) -- (-#1*1.233 , #1*1.7)--(#1*5.727 , #1*1.7)--(#1*5.727 , -#1*0.5)--(-#1*1.233, -#1*0.5);
    \draw[line width=0.8pt] (#1* 2.247,#1* 0.6)circle (#1*0.933);
    \draw[line width=0.8pt] (#1* 0,#1* 0.6)circle (#1*0.933);
    \draw[line width=0.8pt] (#1* 4.494,#1* 0.6)circle (#1*0.933);
  \end{tikzpicture}
}

\newcommand{\cudarcoo}[1]{
  \begin{tikzpicture}
    \draw[line width=0.4pt] (-#1*0.233, -#1*0.5) -- (-#1*0.233 , #1*1.7)--(#1*5.8 , #1*1.7)--(#1*5.8 , -#1*0.5)--(-#1*0.233, -#1*0.5);
    \draw[line width=0.8pt] (#1* 2.247,#1* 0.6)circle (#1*0.933);
    \draw[line width=0.8pt] (#1*0.5,0) arc[start angle=-40, end angle=40, radius=#1*0.933];
    \draw[line width=0.8pt] (#1* 4.5,#1* 0.6)circle (#1*0.933);
  \end{tikzpicture}
}

\newcommand{\cscircle}[1]{
  \begin{tikzpicture}
    \draw[line width=0.4pt] (#1, -#1*0.5) -- (#1 , #1*1.7)--(#1*3.5 , #1*1.7)--(#1*3.5 , -#1*0.5)--(#1, -#1*0.5);
    \draw[line width=0.8pt] (#1* 2.247,#1* 0.6)circle (#1*0.933);
  \end{tikzpicture}
}

\newcommand{\saddle}[1]{
  \begin{tikzpicture}
    \draw[line width=0.8pt] (0,0) arc (140:40:#1cm);
    \draw[line width=0.8pt] (#1 *0.766,#1 *0.3572) -- (#1 *0.766, #1 *0.8428);
    \draw[line width=0.8pt] (0,#1*1.2) arc (220:320:#1cm);
    \draw[line width=0.4pt](#1*0.766,#1*0.6)circle (#1);
  \end{tikzpicture}
}

\newcommand{\lrarc}[1]{
  \begin{tikzpicture}
    \draw[line width=0.8pt] (#1*1.2,0) arc (130:230:#1cm);
    \draw[line width=0.8pt] (0,0) arc (50:-50:#1cm);
    \draw[line width=0.4pt](#1*0.6,-#1*0.766)circle (#1);
  \end{tikzpicture}
}

\newcommand{\clrarc}[1]{
  \begin{tikzpicture}
    \draw[line width=0.8pt] (#1*1.2,0) arc (130:230:#1cm);
    \draw[line width=0.8pt] (0,0) arc (50:-50:#1cm);
    \draw[line width=0.4pt] (-#1*0.5, -#1*1.866) -- (-#1*0.5 , #1*0.333)--(#1*1.7 , #1*0.333)--(#1*1.7 , -#1*1.866)--(-#1*0.5, -#1*1.866);
  \end{tikzpicture}
}

\newcommand{\hsaddle}[1]{
  \begin{tikzpicture}
    \draw[line width=0.8pt] (#1*1.2,0) arc (130:230:#1cm);
    \draw[line width=0.8pt] ( #1 * 0.3572, - #1 * 0.7660)-- (#1 * 0.2 + #1 * 0.6428, - #1 * 0.7660);
    \draw[line width=0.8pt] (0,0) arc (50:-50:#1cm);
    \draw[line width=0.4pt](#1*0.6,-#1*0.766)circle (#1);
  \end{tikzpicture}
}

\newcommand{\tubeab}{\raisebox{-1cm}{\includegraphics[height=2cm]{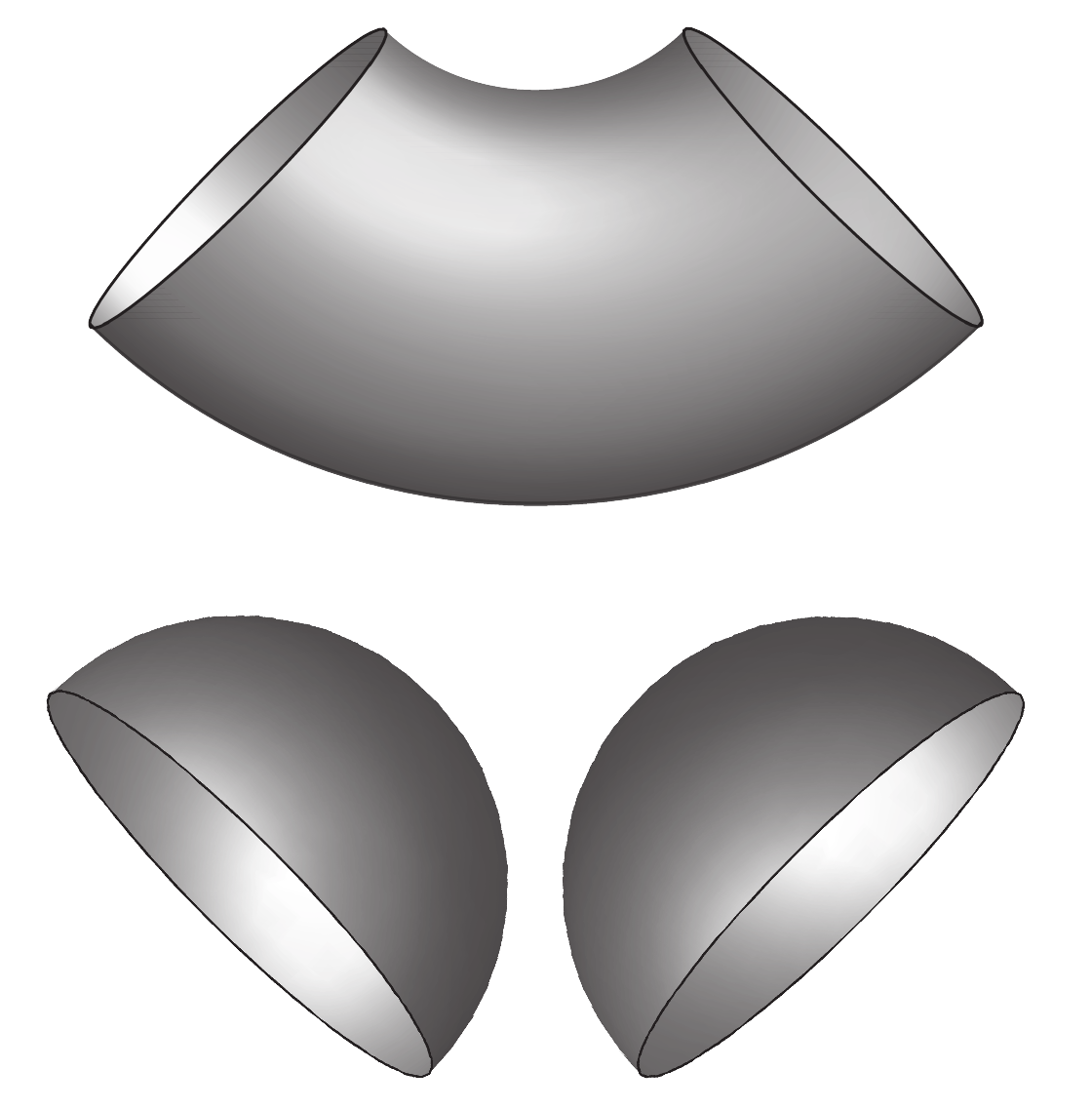}}}
\newcommand{\tubecd}{\raisebox{-1cm}{\includegraphics[height=2cm]{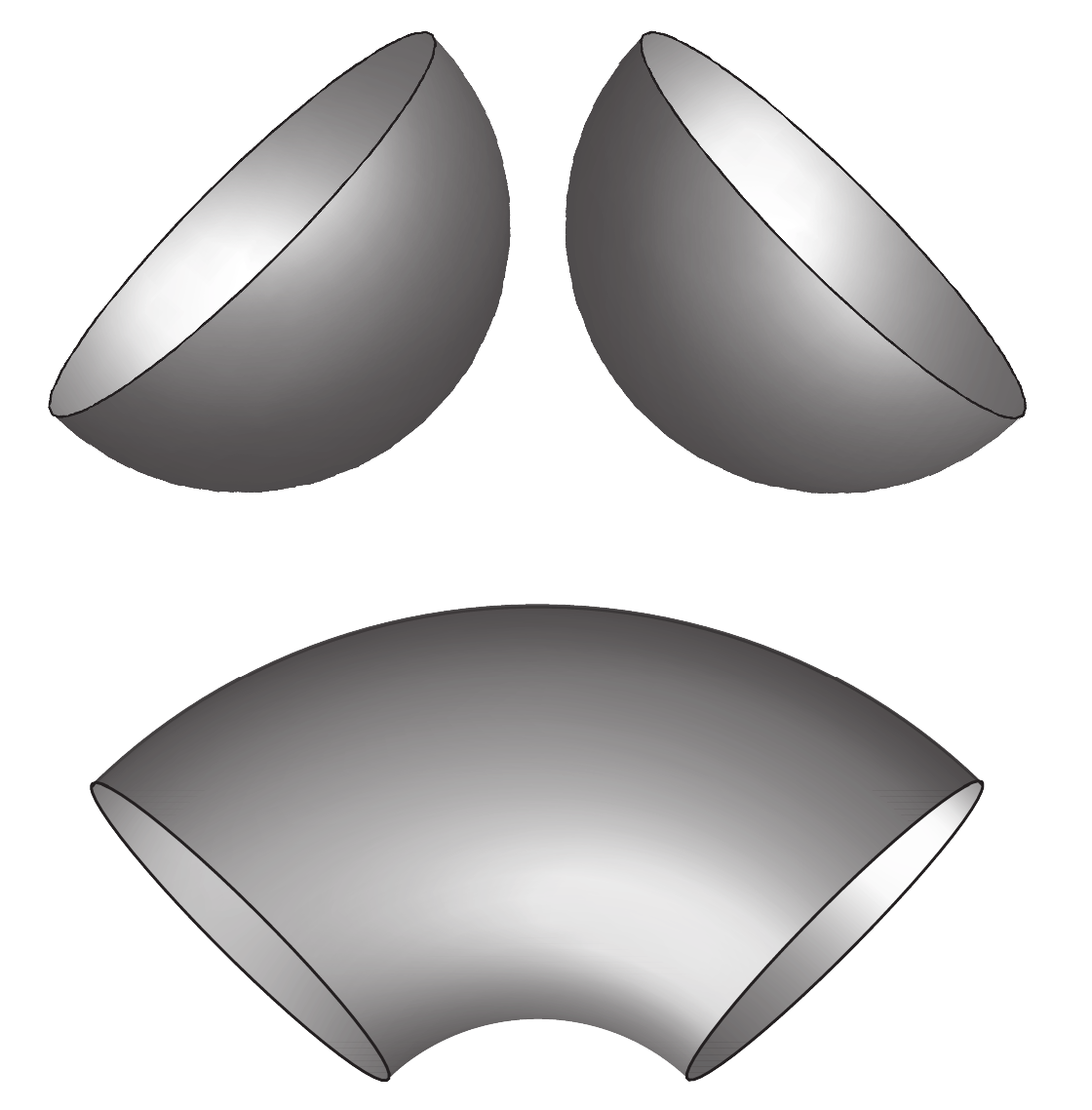}}}
\newcommand{\tubeac}{\raisebox{-1cm}{\includegraphics[height=2cm]{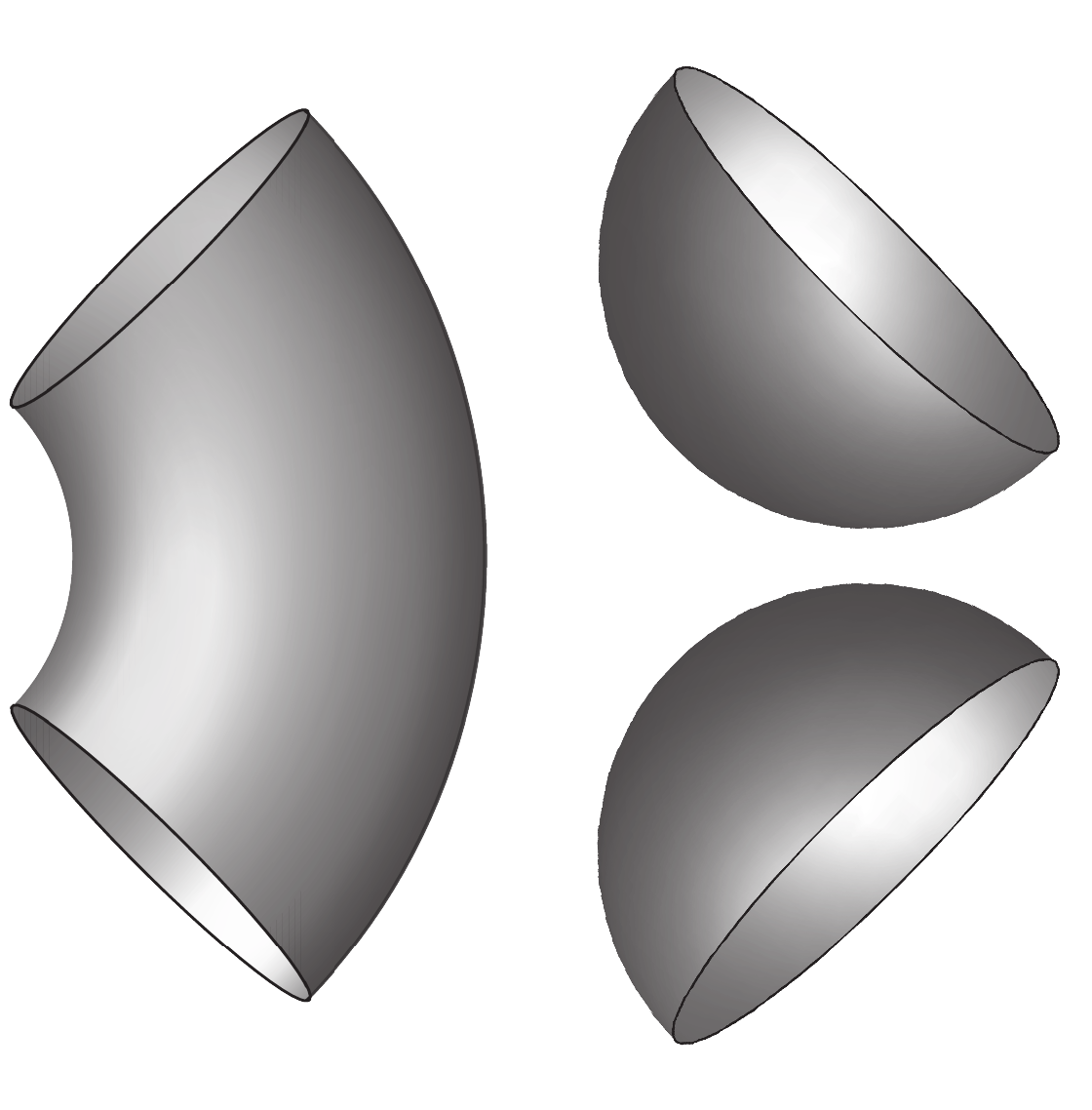}}}
\newcommand{\tubebd}{\raisebox{-1cm}{\includegraphics[height=2cm]{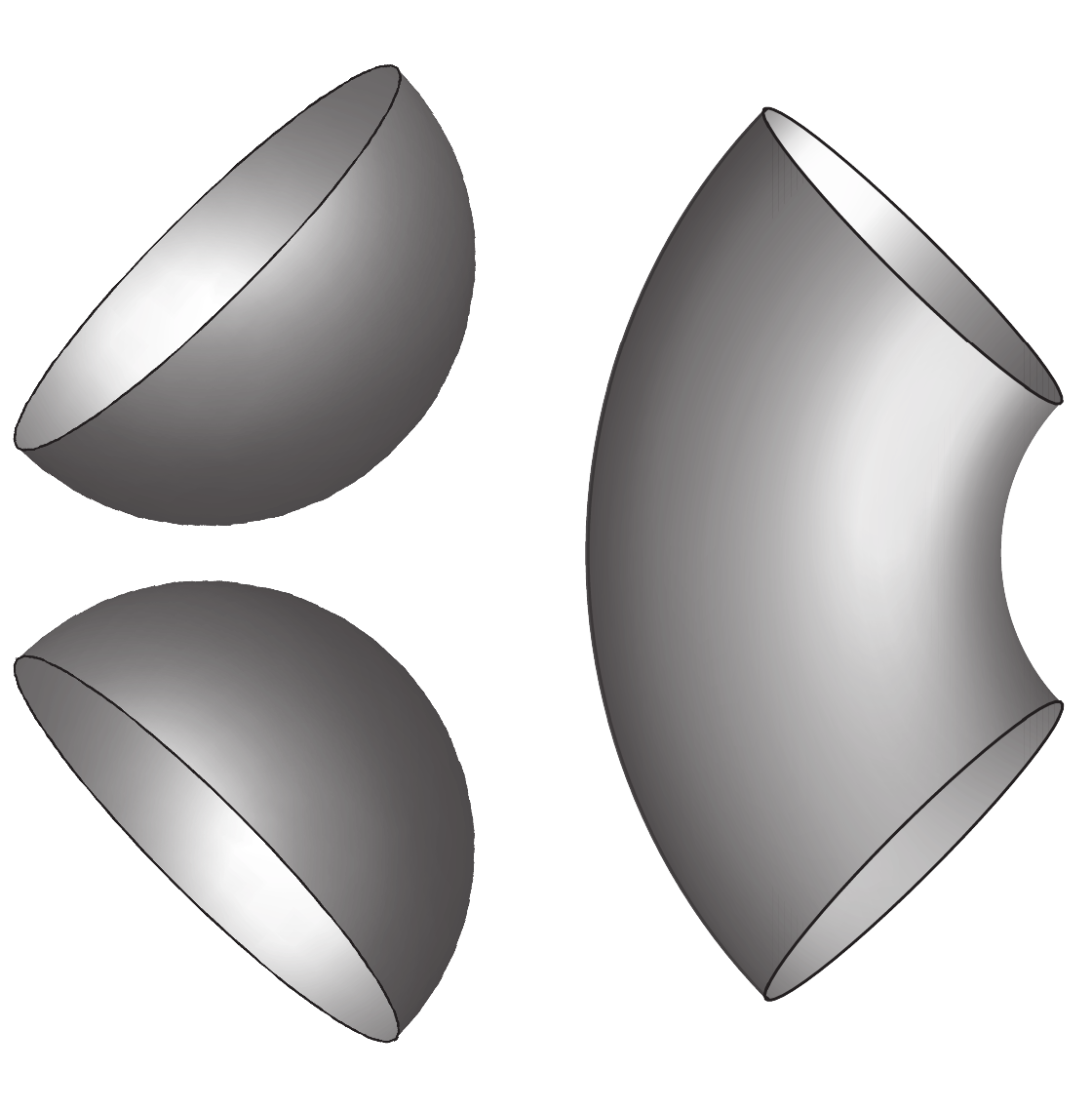}}}

\title{Persistent Khovanov homology of tangles}

\author[1,2]{Jian Liu}
\author[2]{Li Shen}
\author[2,3,4]{Guo-Wei Wei \thanks{Corresponding author: weig@msu.edu}}

\affil[1]{Mathematical Science Research Center, Chongqing University of Technology, Chongqing 400054, China}
\affil[2]{Department of Mathematics, Michigan State University, MI 48824, USA}
\affil[3]{Department of Electrical and Computer Engineering, Michigan State University, MI 48824, USA}
\affil[4]{Department of Biochemistry and Molecular Biology, Michigan State University, MI 48824, USA}

\makeatletter
    \renewcommand*{\@fnsymbol}[1]{\ensuremath{\ifcase#1\or \dagger\or *\or *\or
   \mathsection\or \else\@ctrerr\fi}}
\makeatother
\date{}

\begin{document}
        \maketitle

    \paragraph{Abstract}

    Knot data analysis (KDA), which studies data with curve-type structures such as knots, links, and tangles, has emerging as a promising geometric topology approach in data science. While evolutionary Khovanov homology has been developed to analyze the global persistent topological features of links, it has limitations in capturing the local topological characteristics of knots and links. To address this challenge, we introduce the persistent Khovanov homology of tangles, providing a new mathematical framework for characterizing local features in curve-type data. While tangle homology is inherently abstract, we provide a concrete functor which maps the category of tangles to the category of modules, enabling the computation of tangle homology. Additionally, we employ planar algebra to construct a category of tangles without invoking  fixed boundaries, thereby giving rise to a persistent Khovanov homology functor that enables practical applications. This framework provides a theoretical foundation and practical strategies for the topological analysis of curve-type data.



    \newpage
    \tableofcontents
    \newpage

\section{Introduction}\label{section:introduction}

Curve-type data, including knots, links, and open curves, are ubiquitous in nature, such as river paths, tree rings, DNA packaging, RNA knots, and human vascular networks. Topological analysis of curve-type data presents unique challenges that differ from those encountered in the topological analysis of point cloud data \cite{carlsson2009topology} and data on manifolds \cite{chen2021evolutionary}. Despite significant advancements in geometric topology and knot theory, their application to real-world curve-type data remains largely unexplored.

Khovanov homology, originally developed as a categorification of the Jones polynomial for links, has become a significant knot invariant \cite{bar2002khovanov,khovanov2000categorification}. This framework was later extended to tangles, with distinct constructions provided by Khovanov \cite{khovanov2002functor} and Bar-Natan \cite{bar2005khovanov}. Khovanov homology has been established as a tangle invariant, making it a valuable tool for studying tangles. However, Khovanov homology is not directly suitable for the topological analysis of curve-type data.

Similarly, homology, a fundamental concept in algebraic topology, has limited applicability for topological data analysis. In contrast, persistent homology, a successful algebraic topology approach, has achieved significant success in data science by revealing the underlying topological and geometric structures of point cloud data \cite{carlsson2009topology,edelsbrunner2002topological}. The success of persistent homology is attributed to its multiscale analysis, which creates simplicial complexes at different scales to reveal the "shape" of data. However, persistent homology cannot capture the homotopic shape evolution of data. This limitation was addressed by persistent spectral graph theory \cite{wang2020persistent}, also known as persistent Laplacian, which has demonstrated significant impact in applications \cite{chen2022persistent}. Similarly, persistent Hodge Laplacian, derived from differential geometry, algebraic topology, and multiscale analysis, was introduced for manifold topological analysis \cite{chen2021evolutionary}.

However, these topological methods cannot be directly applied to unveil the multiscale topological features in curve-type data such as knots, links, and tangles, which are mathematical objects in geometric topology. Recently, the multiscale Gauss link integral was introduced in knot data analysis \cite{shen2024knot}, providing an effective framework for capturing the multiscale information of curve-type data. Nonetheless, multiscale Gauss link integral does not preserve topological invariants at small scales. More recently, evolutionary Khovanov homology was introduced to analyze the multi-faceted topological properties of links \cite{shen2024evolutionary}. It is an extension of Khovanov homology for multiscale analysis. Despite that, when studying the Khovanov homology of tangles, the construction based on topological quantum field theory encounters significant challenges. Moreover, the local persistence of tangles does not guarantee that the boundaries of the tangles remain fixed, making it difficult to construct computational tools for existing categories of tangles with fixed boundaries. These challenges call for new strategies to construct tangle persistence for practical applications.

In this work, we introduce the persistent Khovanov homology (PKH) of tangles and present a feasible construction for the practical applications of PKH. Our work overcomes two main technical hurdles. First, the Khovanov complexes of tangles in \cite{bar2005khovanov} are constructed over the additive closure of the category of cobordism (denoted as $\mathrm{Mat}(\mathbf{k}\mathcal{C}ob^{3}_{/l}(B))$, see Section \ref{section:cobordism}), which does not allow for an explicit definition of homology and makes the computation of chain complexes inconvenient. To resolve this problem, we construct a functor that translates the study of cochain complexes of tangles over category $\mathrm{Mat}(\mathbf{k}\mathcal{C}ob^{3}_{/l}(B))$ to the study of cochain complexes of $\mathbf{k}$-modules, thereby facilitating the definition and computation of Khovanov homology in a general setting. Second, the known category of tangles typically requires that tangles have fixed boundaries, which is impractical in applications because maintaining a fixed boundary across a filtration of tangles is rarely feasible. This issue makes it difficult to define the persistence for spatial tangles. To overcome this hurdler, we introduce a new category of tangles based on planar algebra, where morphisms between tangles are interpreted as the inclusion of 1-dimensional manifolds, aligning with the intuitive notion of persistence. By resolving these obstacles, we develop a theory of PKH suitable for practical applications.

In the next section, we review the fundamental concepts related to tangles and the Khovanov homology of tangles. In Section \ref{section:persistence_tangle}, we define the persistent Khovanov homology and introduce a constructive functor on the category of tangles. In the final section, planar tangles are utilized to construct the persistent Khovanov homology of tangles.

\section{Tangle and Khovanov homology}\label{sec:Tangle}

In this section, we review the fundamental concepts and results related to tangles.
We refer to \cite{le1995representation} for basic concepts related to tangles. For the classical theory of Khovanov homology of tangles, we refer to \cite{bar2005khovanov} and \cite{khovanov2000categorification}. Additionally, \cite{khovanov2002functor} explores the homology of $(1,1)$-tangles. Our approach in this work builds upon the relevant theory of the Khovanov homology of tangles as presented in \cite{bar2005khovanov}.

\subsection{Tangle}

A \emph{tangle} is an embedding of finitely many arcs and circles into $\mathbb{R}^2 \times [0,1]$. More precisely, a tangle $T$ is defined as a 1-dimensional compact oriented piecewise smooth submanifold of $\mathbb{R}^3$ lying between two horizontal planes, with every boundary point of $T$ lying on both the top and bottom planes. Another way to describe a tangle is as an embedding of finitely many arcs and circles into a 3-dimensional ball $B^3$, with the ends of the arcs required to lie on the boundary $\partial B^3$ of $B^3$. From now on, we will consider tangles embedded in the 3-dimensional ball $B^3$.
\begin{figure}[htb!]
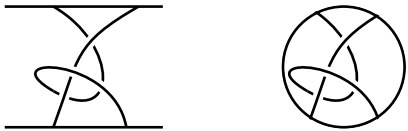

  \centering
  \Tan[1.6]\qquad\qquad\Tano[1.6]\\
  \caption{The tangle representations of a tangle in $\mathbb{R}^2 \times [0,1]$ and $B^3$.}\label{figure:tangle_representation}
\end{figure}

Two tangles $T$ and $T'$ are \emph{isotopic} if there exists a continuous map $H: B^3 \times [0,1] \to B^3$ such that $H(-,0)$ is the identity map, $H(-,1)$ maps $T$ to $T'$, and each map $H(-,t)$ is a homeomorphism that restricts to the identity map on $\partial B^3$.

A \emph{tangle diagram} is a projection $T\to B^2$ of a tangle onto a maximal disk $B^{2}$ in $B^3$ such that it is injective everywhere  except at a finite number of crossing points, which are the projections of only two points of the tangle. A tangle diagram can be seen as a generalization of the concepts of knot diagrams and link diagrams. Two tangle diagrams are equivalent if they are related by a series of Reidemeister moves.
\begin{figure}[htb!]
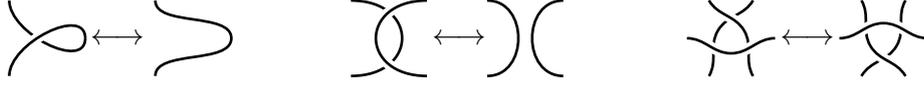

  \centering
  \RIa $\longleftrightarrow$ \RIb[yscale=-1]\qquad\qquad \RIIa$\longleftrightarrow$\RIIb\qquad\qquad \RIIIa$\longleftrightarrow$\RIIIa[rotate=180]
  \caption{The three types of Reidemeister moves.}\label{figure:moves}
\end{figure}

From now on, unless otherwise specified, the tangles considered will always refer to tangle diagrams. For a tangle $T$, we denote the set of crossings of $T$ by $\mathcal{X}(T)$. A crossing of the form $\raisebox{-0.15cm}{\rightcross{0.2}}$ is called an overcrossing, while a crossing of the form $\raisebox{-0.15cm}{\leftcross{0.2}}$ is an undercrossing. Each crossing has a smoothing resolution: $\raisebox{-0.15cm}{\rightcross{0.2}} \Rightarrow \raisebox{-0.15cm}{\udarc{0.3}} + \raisebox{-0.15cm}{\lrarc{0.3}}$ or $\raisebox{-0.15cm}{\leftcross{0.2}} \Rightarrow \raisebox{-0.15cm}{\udarc{0.3}} + \raisebox{-0.15cm}{\lrarc{0.3}}$. Here, $\raisebox{-0.15cm}{\udarc{0.3}}$, called the 0-smoothing, is the tangle obtained by locally changing a crossing into two opposing arcs, one above the other. Similarly, $\raisebox{-0.15cm}{\lrarc{0.3}}$, called the 1-smoothing, is obtained by locally changing a crossing into two opposing arcs, one to the left and one to the right. In this work, the 0-smoothings and 1-smoothings are always conducted on the undercrossing $\raisebox{-0.15cm}{\leftcross{0.2}}$.
Let $n = |\mathcal{X}(T)|$ be the number of crossings of $T$. Then there are $2^n$ states of the smoothing resolution of $T$. The $2^n$ states form a state cube $\{0,1\}^{n}$. Each vertex represents a state of the smoothing resolution and can be described by a sequence $(s_{i})_{0\leq i\leq n}\in \{0,1\}^{n}$ of 0s and 1s of length $n$. Each edge represents two state sequences that differ in exactly one position. For an oriented tangle, we have the right-handed crossing $\raisebox{-0.15cm}{\arightcross{0.2}}$ and the left-handed crossing $\raisebox{-0.15cm}{\aleftcross{0.2}}$. We always assign the symbol $+$ to the right-handed crossing and the symbol $-$ to the left-handed crossing. Let $n_{+}$ denote the number of right-handed crossings, and let $n_{-}$ denote the number of left-handed crossings.

The study of the category of tangles involves the 2-category structure of tangles, which has been developed in   \cite{baez2003higher,fischer19942,langford19972}. Roughly speaking, this category has the boundaries of tangles as objects, tangles as 1-morphisms, and cobordisms connecting tangles as 2-morphisms. The 2-morphisms, depicted by movies, are generated by a family of moves as detailed in \cite{carter1993reidemeister, roseman1998reidemeister}. In particular, the edges of the state cube can be characterized by a cobordism between the smoothings of a tangle.

\subsection{Cobordism and bracket complex}\label{section:cobordism}

Let $M$ and $N$ be two compact manifolds without boundary. A \emph{cobordism} $\Sigma$ between $M$ and $N$ is a compact manifold with boundary such that its boundary is the disjoint of $M$ and $N$, $\partial \Sigma=M\sqcup N$.

Given a tangle $T$, recall that we can obtain a state cube $\{0,1\}^{n}$.  Each vertex of the cube represents a tangle with the boundary $\partial T$. Moreover, there is a cobordism connecting the tangles corresponding to the end vertices of an edge of the state cube.
Considering such tangles corresponding to some smoothing of $T$ as objects, and the cobordisms between these tangles as morphisms, we obtain a category $\mathcal{C}ube(T)$.
Generally, for a finite set of points $B$ on a circle, we have a category $\mathcal{C}ob^3(B)$, whose objects are the tangles corresponding to some smoothing of a tangle, and whose morphisms are the cobordisms between such tangles. For a fixed tangle $T$, the category $\mathcal{C}ube(T)$ is a subcategory of $\mathcal{C}ob^3(\partial T)$.

Let $\mathbf{k}$ be a commutative ring with a unit. One can extend $\mathcal{C}ob^3(B)$ to a pre-additive category $\mathbf{k}\mathcal{C}ob^3(B)$ as follows. The objects in $\mathbf{k}\mathcal{C}ob^3(B)$ are the same as the objects in $\mathcal{C}ob^3(B)$, and the morphisms in $\mathbf{k}\mathcal{C}ob^3(B)$ are linear combinations of morphisms in $\mathcal{C}ob^3(B)$. That is, the set $\mathrm{Hom}_{\mathbf{k}\mathcal{C}ob^3(B)}(T,T')$ is a $\mathbf{k}$-module generated by the morphisms in the set $\mathrm{Hom}_{\mathcal{C}ob^3(B)}(T,T')$ of morphisms from $T$ to $T'$ for any objects $T$ and $T'$ in $\mathcal{C}ob^3(B)$.

\begin{definition}
For a pre-additive category $\mathcal{C}$, we can define a category $\mathrm{Mat}_{\mathbf{k}}(\mathcal{C})$ with:
\begin{itemize}
  \item Objects of the form $\mathcal{O} = \bigoplus\limits_{i=1}^{m} \mathcal{O}_{i}$ for $\mathcal{O}_{i} \in \mathcal{C}$.
  \item Morphisms that are matrices of the form $f = (f_{ij})_{i,j} : \bigoplus\limits_{i=1}^{m} \mathcal{O}_{i} \to \bigoplus\limits_{j=1}^{k} \mathcal{O}_{j}'$, where $f_{ij} : \mathcal{O}_{i} \to \mathcal{O}_{j}'$ are morphisms in $\mathcal{C}$ for $1 \leq i \leq m$ and $1 \leq j \leq k$.
  \item Composition of morphisms given by matrix multiplication.
\end{itemize}
\end{definition}
The construction $\mathrm{Mat}_{\mathbf{k}}(\mathcal{C})$ is an additive category, which is the additive closure of the category $\mathcal{C}$. Furthermore, one can define a cochain complex in an additive category.
\begin{definition}
Let $\mathcal{C}$ be an additive category. The category $\mathbf{Ch}^{\bullet}(\mathcal{C})$ of cochain complexes over $\mathcal{C}$ is defined as follows. Its objects are of the form
\begin{equation*}
  \xymatrix{
  \cdots\ar@{->}[r]&\Omega^{r-1}\ar@{->}[r]^{d^{r-1}}&\Omega^{r}\ar@{->}[r]^{d^{r}}&\Omega^{r+1}\ar@{->}[r]&\cdots
  }
\end{equation*}
such that $d^{r+1}\circ d^{r}=0$ for any $r$, and its morphisms are of the form $f^{r}:(\Omega^{r}_{a},d_{a})\to (\Omega^{r}_{b},d_{b})$ such that $f^{r-1}\circ d_{a}=d_{b}\circ f^{r}$ for any $r$.
\end{definition}

Let $T$ be a tangle with $n$ crossings. The state cube associated with $T$ has vertices indexed by states $s = (s_i)_{0 \leq i \leq n} \in \{0,1\}^{n}$, where each $s_i$ represents a smoothing choice at the $i$-th crossing of the tangle. For a given state $s$, we denote $\ell(s) = \sum\limits_{i=1}^{n} s_i$. Next, for the smoothing tangle $T_s$ corresponding to state $s$, we assign a height function $h(s) = \ell(s) - n_{-}$, where $n_{-}$ is the number of left-handed crossings in the original tangle $T$. This height measures the relative position of each smoothing state in the cube. Recall that the category $\mathcal{C}ube(T)$ is a subcategory of $\mathcal{C}ob^3(\partial T)$. We have a graded object in $\mathrm{Mat}(\mathbf{k}\mathcal{C}ob^3(B))$ given by
\begin{equation*}
    \xymatrix{
  \cdots\ar[r]&[[T]]^{k-1}\ar[r]^-{d^{k-1}}&[[T]]^{k}\ar[r]^-{d^{k}}&[[T]]^{k+1}\ar[r]^-{d^{k+1}}&\cdots,
  }
\end{equation*}
where each graded piece $[[T]]^{k} = \bigoplus\limits_{\ell(s) = k} T_s$ is a direct sum over all smoothing tangles $T_s$ whose length $\ell(s) = k$. The morphism $d^{k}$ is given by
\begin{equation*}
  d^k = \sum_{\xi} (-1)^{\mathrm{sgn}(\xi)} d_{\xi} : [[T]]^k \to [[T]]^{k+1},
\end{equation*}
where the sum is over all edges $\xi = (\xi_1, \dots, \xi_{i-1}, \star, \xi_{i+1}, \dots, \xi_{|\mathcal{X}(T)|}) \in \{0, 1, \star\}^{|\mathcal{X}(T)|}$ in the state cube that connect a state $s$ with a neighboring state $s'$ that differs by one position. Here, $\xi_{j} \in \{0, 1\}$ for $j\neq i$ and $\star$ indicates an edge connecting $0$ to $1$.
The map $d_{\xi}$ denotes the cobordism morphism between the smoothing tangles $T_s$ and $T_{s'}$. The sign $\mathrm{sgn}(\xi)$ is determined by the number of $1$s in $\xi$ that appear before the first $\star$.

Note that the cube $\mathcal{C}ube(T)$ is anti-commutative. This means that for each face of the cube, represented by the following diagram:
\begin{equation*}
  \xymatrix{
  T_{s}\ar[r]^{d_{\xi}}\ar[d]_{d_{\eta}}&T_{s'}\ar[d]^{d_{\eta'}}\\
  T_{\tilde{s}}\ar[r]^{d_{\tilde{\xi}}}&T_{\tilde{s}'}
  }
\end{equation*}
we have the anti-commutativity relation $d_{\widetilde{\xi}} \circ d_{\eta} = -d_{\eta'} \circ d_{\xi}$. This condition ensures that the composition of differentials along the edges of each face of the state cube satisfies the appropriate signs, maintaining the structure of a cochain complex.

\begin{proposition}[\text{\cite[Proposition 3.4]{bar2005khovanov}}]
The construction $([[T]]^{\ast},d^{\ast})$ above is a cochain complex over $\mathrm{Mat}(\mathbf{k}\mathcal{C}ob^3(\partial T))$.
\end{proposition}

The cochain complex $([[T]]^{\ast}, d^{\ast})$ is called the \emph{bracket complex} of $T$. However, the bracket complex $([[T]]^{\ast}, d^{\ast})$ is not a tangle invariant in the category $\mathbf{Ch}^{\bullet}(\mathrm{Mat}(\mathbf{k}\mathcal{C}ob^3(\partial T)))$ of cochain complexes over $\mathrm{Mat}(\mathbf{k}\mathcal{C}ob^3(\partial T))$.
In \cite{bar2005khovanov}, Bar-Natan obtains a new category from $\mathrm{Mat}(\mathbf{k}\mathcal{C}ob^3(\partial T))$ by modding out some equivalence relations. In this new category, he proves that the bracket complex is a tangle invariant up to chain homotopy.

Let $B$ be a finite set of points on a circle. The category $\mathbf{k}\mathcal{C}ob_{/l}^3(B)$ is a localization of the category $\mathbf{k}\mathcal{C}ob^3(B)$ defined as follows. The objects are the same as the objects in $\mathbf{k}\mathcal{C}ob^3(B)$. The morphisms are those of $\mathbf{k}\mathcal{C}ob^3(B)$ under the following equivalence relations:
\begin{itemize}
  \item[$(S)$] $C +  S^{2} = 0$ for any cobordism $C$ in $\mathbf{k}\mathcal{C}ob^3(B)$. Here, $S^{2}$ is the cobordism of the 2-dimensional sphere.
  \item[$(T)$] $C + T^{2} = 2C$ for any cobordism $C$ in $\mathbf{k}\mathcal{C}ob^3(B)$. Here, $T^{2}$ is the cobordism corresponding to the torus.
  \item[$(4Tu)$] $C_{12} + C_{34} = C_{13} + C_{24}$. Here, $C$ is a cobordism whose intersection with a ball is the union of four disks $D_{i}$, $i = 1, 2, 3, 4$, and $C_{ij}$ is the cobordism obtained by removing $D_{i}$ and $D_{j}$ from $C$ and replacing them with a tube that has the same boundary.
\begin{figure}[htb!]
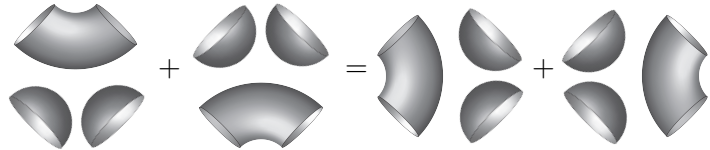

  \centering
  $\tubeab+\tubecd=\tubeac+\tubebd$
  \caption{The cobordism representation of the $(4Tu)$ relation.}\label{figure:4tu}
\end{figure}
\end{itemize}
Since $\mathbf{k}\mathcal{C}ob^3(B)$ is a pre-additive category, so is $\mathbf{k}\mathcal{C}ob_{/l}^3(B)$. Moreover, one has an additive category $\mathrm{Mat}(\mathbf{k}\mathcal{C}ob_{/l}^3(B))$.

\begin{theorem}[\text{\cite[Theorem 1]{bar2005khovanov}}]
The construction $([[T]]^{\ast},d^{\ast})$ is a tangle invariant up to chain homotopy in the category $\mathbf{Ch}^{\bullet}(\mathrm{Mat}(\mathbf{k}\mathcal{C}ob^3_{/l}(\partial T)))$ of cochain complexes over $\mathrm{Mat}(\mathbf{k}\mathcal{C}ob^3_{/l}(\partial T))$.
\end{theorem}
The above theorem says that the bracket complex of $T$ in the category of cochain complexes over $\mathrm{Mat}(\mathbf{k}\mathcal{C}ob^3_{/l}(\partial T))$ is an invariant  under Reidemeister moves up to chain homotopy.

\begin{definition}
Let $T$ be a tangle. The \emph{Khovanov complex} of $T$ is the cochain complex $(Kh^{\ast}(T),d_{T}^{\ast})$ given by $Kh^{p}(T)=[[T]]^{p-n_{-}}$ and $d^{p}_{T}=d^{p-n_{-}}$.
\end{definition}

The Khovanov complex and the bracket complex differ by a height shift. Specifically, when the tangle $T$ is a knot or link, the corresponding Khovanov complex is consistent with the Khovanov complex of the knot or link. Similarly, if two tangles $T_{1}$ and $T_{2}$ differ by some Reidemeister moves, there exists a chain homotopy equivalence $Kh(T_{1}) \simeq Kh(T_{2})$.

Let $B\subseteq S^{1}$ be a finite set of points. Let $\mathcal{C}ob^{4}(B)$ be the category whose objects are tangles in a disk $D$ with boundary $B$, and whose morphisms are 2-dimensional cobordisms between these tangles in $D\times [-\epsilon,\epsilon]\times [0,1]$ with boundary $B\times [-\epsilon,\epsilon]\times [0,1]$. The construction $Kh$ gives a functor $Kh_{B} : \mathcal{C}ob^{4}(B) \to \mathbf{Ch}^{\bullet}(\mathrm{Mat}(\mathbf{k}\mathcal{C}ob^3_{/l}(B)))$ from the category $\mathcal{C}ob^{4}(B)$ of tangles with boundary $B$ to the category of cochain complexes over $\mathrm{Mat}(\mathbf{k}\mathcal{C}ob^3_{/l}(B))$.

\begin{theorem}
The functor $Kh_{B}:\mathcal{C}ob^{4}(B) \to \mathbf{Ch}^{\bullet}(\mathrm{Mat}(\mathbf{k}\mathcal{C}ob^3_{/l}(B)))$ maps the equivalence classes of isotopy of tangles to the equivalence classes of chain homotopy of cochain complexes.
\end{theorem}
It is worth noting that  Bar-Natan's construction directly forms cochain complexes in the category $\mathbf{k}\mathcal{C}ob^3_{/l}(B)$, which provides a more fundamental approach compared to the Khovanov complex constructed within the framework of topological quantum field theory (TQFT). However, this more intrinsic construction comes with a significant limitation: we cannot directly define Khovanov homology because the category $\mathbf{k}\mathcal{C}ob^3_{/l}(B)$ is not an abelian category.

\subsection{Khovanov homology of tangles}\label{section:tqft}

Let $\mathcal{A}b$ be an abelian category. Note that any functor $\mathcal{F}:\mathbf{k}\mathcal{C}ob^3_{/l}(B)\to \mathcal{A}b$ can extend to a functor $\mathcal{F}:\mathrm{Mat}(\mathbf{k}\mathcal{C}ob^3_{/l}(B))\to \mathcal{A}b$. Thus one can obtain a functor $\mathcal{F}^{\bullet}:\mathbf{Ch}^{\bullet}(\mathrm{Mat}(\mathbf{k}\mathcal{C}ob^3_{/l}(B)))\to \mathbf{Ch}^{\bullet}(\mathcal{A}b)$ given by $\mathcal{F}^{\bullet}(\Omega^{\ast},d^{\ast})=(\mathcal{F}\Omega^{\ast},\mathcal{F}d^{\ast})$. Recall that the homology is a functor $H:\mathbf{Ch}^{\bullet}(\mathcal{A}b)\to \mathcal{A}b$ from the category of cochain complexes to an abelian category. We have the definition of Khovanov homology of tangles as follows.

\begin{definition}
Let $B$ be a finite set of points on a circle. Let $\mathcal{F}:\mathbf{k}\mathcal{C}ob^3_{/l}(B)\to \mathcal{A}b$ be a functor into an abelian category. The \emph{Khovanov homology of tangles} with respect to $\mathcal{F}$ is the composition of functors
\begin{equation*}
  \xymatrix{
  \mathcal{C}ob^{4}(B)\ar@{->}[r]^-{Kh_{B}}&\mathbf{Ch}^{\bullet}(\mathrm{Mat}(\mathbf{k}\mathcal{C}ob^3_{/l}(B))\ar@{->}[r]^-{\mathcal{F}^{\bullet}}&\mathbf{Ch}^{\bullet}(\mathcal{A}b)\ar@{->}[r]^-{H}& \mathcal{A}b.
  }
\end{equation*}

\end{definition}

It can be verified that $H\mathcal{F}^{\bullet}Kh_{B}$ is an isotopy invariant of tangles with boundary $B$.
The definition of Khovanov homology mentioned above relies on the functor $\mathcal{F}$. Recall that the category $\mathcal{M}od_{\mathbf{k}}$ of modules is an abelian category. In TQFT, there is a standard construction of the functor $\mathcal{F} : \mathcal{C}ob^3(\emptyset) \to \mathcal{M}od_{\mathbf{k}}$, which yields the usual definition of Khovanov homology of links.

\begin{wrapfigure}{r}{0.3\textwidth}
    \centering
    \includegraphics[width=0.3\textwidth]{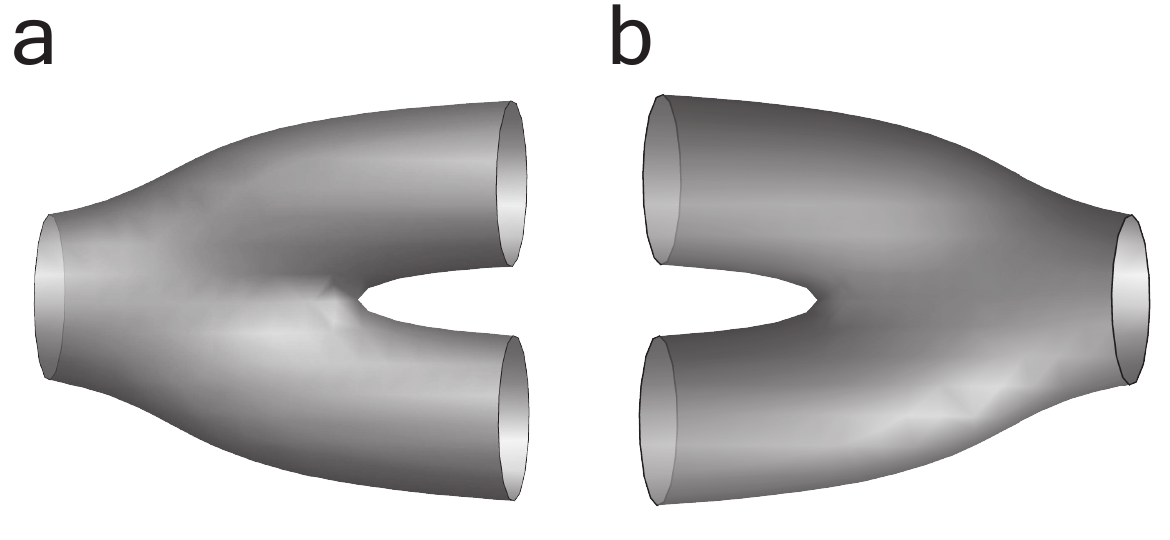} 
    \caption{The cobordisms corresponding to the maps $\wedge$ and $\vee$.}
    \label{fig:product}
\end{wrapfigure}
Consider the functor $\mathcal{F}:\mathcal{C}ob^3(\emptyset)\to \mathcal{M}od_{\mathbf{k}}$ constructed as follows. Let $V$ be a $\mathbf{k}$-module generated by the elements $v_{+}$ and $v_{-}$.
For a link $L$, the $\mathbf{k}$-module $\mathcal{F}(L)$ is the tensor product $\underbrace{V\otimes_{\mathbf{k}} V\otimes_{\mathbf{k}}\cdots \otimes_{\mathbf{k}} V}_{r(L)}$, where $r(L)$ is the number of circles of $L$. Note that the morphisms in $\mathcal{C}ob^3(\emptyset)$ are compositions of those represented by cap and cup cobordisms, along with the morphisms $\wedge$ and $\vee$, corresponding to saddle cobordisms. The functor $\mathcal{F}$ is given as follows:
\begin{equation*}
  \mathcal{F}(\bigcap)=\epsilon:\mathbf{k}\to V,\quad 1\mapsto v_{+},
\end{equation*}
where $\bigcap : \emptyset \to \bigcirc$ denotes the morphism corresponding to the cap cobordism shown in Figure \ref{fig:capcup}a,
\begin{equation*}
  \mathcal{F}(\bigcup)=\eta:V\to \mathbf{k},\quad v_{+}\mapsto 0,v_{-}\mapsto 1
\end{equation*}
where $\bigcup : \bigcirc \to \emptyset$ denotes the morphism corresponding to the cup cobordism depicted in Figure \ref{fig:capcup}b.
\begin{equation*}
  \mathcal{F}(\wedge) = \Delta: V \to V \otimes_{\mathbf{k}} V, \quad \Delta:\left\{
                             \begin{array}{ll}
                               v_{+} \mapsto v_{+} \otimes v_{-} + v_{-} \otimes v_{+}, \\
                               v_{-} \mapsto v_{-} \otimes v_{-},
                             \end{array}
                           \right.
\end{equation*}
where $\wedge$ denotes the splitting of a circle into two circles, as shown in Figure \ref{fig:product}a.
\begin{equation*}
  \mathcal{F}(\vee) = m: V \otimes_{\mathbf{k}} V \to V, \quad m:\left\{
                             \begin{array}{ll}
                               v_{+} \otimes v_{+} \mapsto v_{+}, \quad v_{-} \otimes v_{+} \mapsto v_{-}, \\
                               v_{+} \otimes v_{-} \mapsto v_{-}, \quad v_{-} \otimes v_{-} \mapsto 0,
                             \end{array}
                           \right.
\end{equation*}
where $\vee$ denotes the merging of circles into a circle, as shown in Figure \ref{fig:product}b.
One can verify that the construction above can extend to a functor $\mathcal{F}:\mathbf{k}\mathcal{C}ob^3_{/l}(\emptyset)\to \mathcal{M}od_{\mathbf{k}}$. Fix a link $L$, the construction $\mathcal{F}^{\bullet}Kh_{\emptyset}(L)$ is a cochain complex of $\mathbf{k}$-modules. The differential is the $\mathbf{k}$-module homomorphism $d^{k} = \sum\limits_{\xi} (-1)^{\sgn(\xi)} \mathcal{F}(d_{\xi})$, where $d_{\xi}$ is the map given by $\mathcal{F}(\vee)$ or $\mathcal{F}(\wedge)$ on the components involved in merging or splitting, and the identity on other components. In this case, the homology $H(\mathcal{F}^{\bullet}Kh_{\emptyset}(L))$ coincides with the classical definition of Khovanov homology for links.
Additionally, each element $x$ in the cochain complex $\mathcal{F}^{\bullet}Kh_{\emptyset}(L)$ has a quantum grading given by $\Phi(x) = p(x) + n_{+} - n_{-} + \theta(x)$, where $p(x)$ is the height of $x$ in the cochain complex, and $\theta(x)$ is obtained by taking $\theta(v_{+}) = 1$ and $\theta(v_{-}) = -1$.

In the remainder of this paper, we will denote $\mathcal{K}_{B} = \mathcal{F}^{\bullet}Kh_{B} : \mathcal{C}ob^4(B) \to \mathbf{Ch}^{\bullet}(\mathcal{M}od_{\mathbf{k}})$ and $H(-;\mathcal{F}) = H\mathcal{K}_{B} : \mathcal{C}ob^4(B)  \to \mathcal{M}od_{\mathbf{k}}$ for simplicity. Unless otherwise specified, the notation $\mathcal{K}_{\emptyset} = \mathcal{F}^{\bullet}Kh_{\emptyset} : \mathcal{C}ob^4(\emptyset) \to \mathbf{Ch}^{\bullet}(\mathcal{M}od_{\mathbf{k}})$ will always be based on the construction of $\mathcal{F}$ given in Section \ref{section:tqft}.

Now, consider the case where $\mathbf{k}$ is a field. Let $M = \bigoplus\limits_{i \in \mathbb{Z}} M_{i}$ be a graded $\mathbf{k}$-linear space. The \emph{graded dimension} of $M$ is defined as the polynomial $\qdim M = \sum\limits_{i \in \mathbb{Z}} q^{i} \dim M_{i}$ in the variable $q$.
For the above construction of $V$, let $\deg v_{+} = 1$ and $\deg v_{-} = -1$. Then we have $\qdim V = q + q^{-1}$. The \emph{graded Euler characteristic} of a cochain complex $C^{\ast}$ of $\mathbf{k}$-linear spaces is defined by $\mathcal{X}_{q} = \sum\limits_{k} (-1)^{k} \qdim C^{k}$. Let $T$ be a tangle. Then the Jones polynomial of $T$ can be expressed as
\begin{equation*}
  J_{q}(T) = \sum\limits_{k} (-1)^{k} \qdim H^{k}(T, \mathcal{F}).
\end{equation*}
When $\partial T = \emptyset$, $J_{q}(T)$ corresponds to the classical unnormalized Jones polynomial.

\section{Persistent Khovanov homology of tangles  }\label{section:persistence_tangle}

Tangles are common research objects across various disciplines, such as curve-like data which locally appear as tangles, and they have significant application potential. Studying the persistent Khovanov homology of tangles   is a natural idea, and it offers a new tool for understanding complex entangled structures. In this section, we introduce the concept of persistent Khovanov homology of tangles. Moreover, to ensure the computability of tangle homology, we provide a construction from the category $\mathcal{C}ob^{3}(B)$  of tangles to the category of $\mathbf{k}$-modules.

\subsection{Persistent Khovanov homology  }\label{section:persistence_tangle}
Let $B$ be a finite set of points on the circle $S^{1}$. Suppose $(X,\leq)$ is a poset with the partial order $\leq$. Then $(X,\leq)$ can be regarded as a category whose objects are the elements in $X$, and whose morphisms are the pairs $x\leq x'$ with $x,x'\in X$.

\begin{definition}
A \emph{persistence tangle} with boundary $B$ is a functor $\mathcal{P}:(X,\leq)\to \mathcal{C}ob^{4}(B)$ into the category of tangles with boundary $B$.
\end{definition}

\begin{example}
A movie of a tangle cobordism $\Sigma$ is the intersection of the tangle cobordism in $D \times [-\epsilon, \epsilon] \times [0,1]$ with cylinder spaces $D \times [-\epsilon, \epsilon] \times \{t\}$. This movie is called the \textit{movie representation} of the tangle cobordism $\Sigma$. For each $t \in [0,1]$, the intersection corresponds to a tangle. The movie representation of a tangle cobordism can be understood as depicting each frame of the movie.

A movie representation of the tangle cobordism in the category $\mathcal{C}ob^{4}(B)$ can equivalently be described as a persistence tangle. Given a tangle cobordism $\Sigma$ with boundary $B\times [-\epsilon,\epsilon]\times [0,1]$, the functor $\mathcal{P}:([0,1],\leq)\to \mathcal{C}ob^{4}(B)$ given by $\mathcal{P}(t)=\Sigma\cap (D \times [-\epsilon, \epsilon] \times \{t\})$ is a persistence tangle. The persistence tangle $\mathcal{P}(t)$ is also a movie representation of the tangle cobordism $\Sigma$.
\end{example}

\begin{definition}
Let $\mathcal{P}:(X,\leq)\to \mathcal{C}ob^{4}(B)$ be a persistence tangle. The \emph{persistent homology of $\mathcal{P}$} is the composition of functors
\begin{equation*}
  \xymatrix{
  (X,\leq)\ar@{->}[r]^-{\mathcal{P}}&\mathcal{C}ob^{4}(B)\ar@{->}[r]^{H(-;\mathcal{F})}& \mathcal{M}od_{\mathbf{k}}.
  }
\end{equation*}
Here, $H(-;\mathcal{F}) :\mathcal{C}ob^{4}(B) \to \mathcal{M}od_{\mathbf{k}}$ is the homology of tangles.
\end{definition}
Specifically, for any $a \leq b$ in $X$, the $(a,b)$-persistent Khovanov homology of the persistence tangle $\mathcal{P}:(X,\leq)\to \mathcal{C}ob^{4}(B)$ is given by
\begin{equation*}
  H_{a,b}^{p}(\mathcal{P},B) = \im \left(H^{p}(\mathcal{P}(a),B) \to H^{p}(\mathcal{P}(b),B)\right), \quad p \in \mathbb{Z}.
\end{equation*}
The graded dimension of $H_{a,b}^{p}(\mathcal{P},B)$ is the Betti polynomial $\beta_{a,b}^{p}(q)=\sum\limits_{\omega\in H_{a,b}^{p}(\mathcal{P},B)}q^{\Phi(\omega)}$, where $\Phi(\omega)$ is the quantum grading of $\omega$.

Specifically, let $(X, \leq) = (\mathbb{Z}, \leq)$. Let $\mathbf{H} = \bigoplus\limits_{a \in \mathbb{Z}} H^{\ast}(\mathcal{P}(a), B)$. For any $a \in \mathbb{Z}$, we have a map
\begin{equation*}
  z: H^{\ast}(\mathcal{P}(a), B) \to H^{\ast}(\mathcal{P}(a + 1), B),
\end{equation*}
which induces a map $z: \mathbf{H} \to \mathbf{H}$. Thus, $\mathbf{H}$ is a $\mathbf{k}[z]$-module. This implies that the persistent Khovanov homology of tangles is also a persistence module.
Under certain conditions, persistent Khovanov homology exhibits the structure theorem of persistence modules, the fundamental characterization of the corresponding barcodes, as well as the stability theorem for persistence modules. We shall not expend further in elaborating on these analogous results.
\begin{figure}[htb!]
  \centering
  \includegraphics[width=0.6\textwidth]{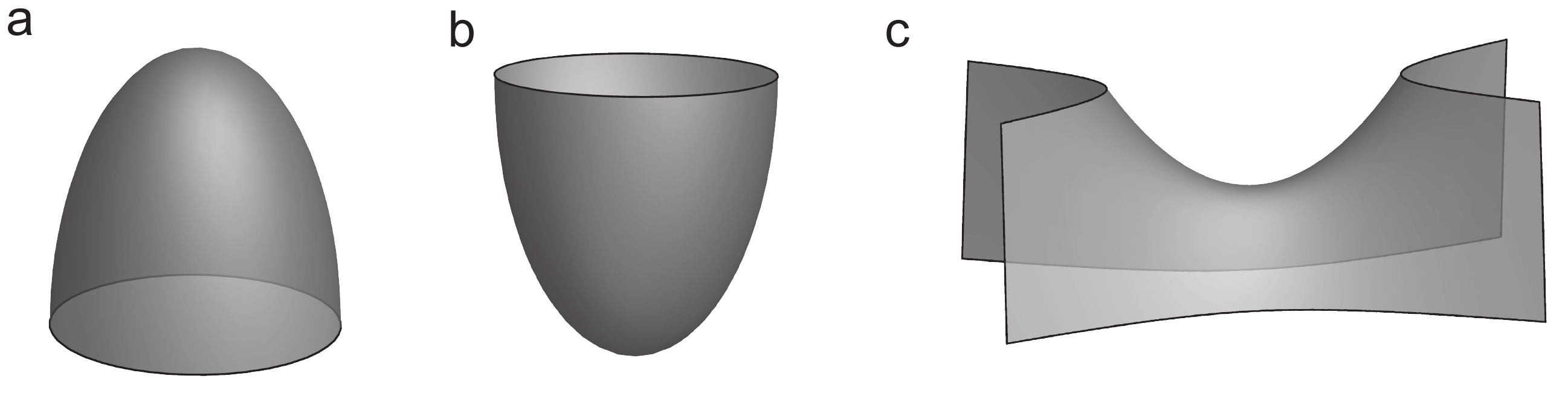}\\
  \caption{The subfigures a, b, and c represent the cap cobordism, cup cobordism, and saddle cobordism, respectively.}\label{fig:capcup}
\end{figure}

Let $\raisebox{-0.15cm}{\saddle{0.3}} : \raisebox{-0.15cm}{\udarc{0.3}}\to \raisebox{-0.15cm}{\lrarc{0.3}}$ be the morphism representing the saddle cobordism (see Figure \ref{fig:capcup}c).
It is known that the morphisms in the category $\mathcal{C}ob^{4}(\emptyset)$ are generated by the three Reidemeister moves and the morphisms $\bigcap$, $\bigcup$, and $\raisebox{-0.15cm}{\saddle{0.3}}$.

Let $\bigcap : T \to T \amalg \bigcirc$ be the morphism of tangles that produces a circle. Here, $T \amalg \bigcirc$ denotes the disjoint union of the tangle $T$ and the circle $\bigcirc$.
Note that the cochain complex $\mathcal{K}_{\emptyset}(T \amalg \bigcirc) = \mathcal{K}_{\emptyset}(T) \otimes_{\mathbf{k}} V$. Thus, the morphism $\mathcal{K}_{\emptyset}(\bigcap) : \mathcal{K}_{\emptyset}(T) \to \mathcal{K}_{\emptyset}(T) \otimes_{\mathbf{k}} V$ is given by $\mathcal{K}_{\emptyset}(\bigcap)(x) = x \otimes v_{+}$. Therefore, the corresponding persistent Khovanov homology of $\bigcap$ is
\begin{equation*}
    \operatorname{im} H^{\ast}(\bigcap;\mathcal{F}) = H^{\ast}(T;\mathcal{F}) \otimes v_{+}.
\end{equation*}
Let $\bigcup : T \amalg \bigcirc \to T $ be the morphism of tangles corresponding to the cup cobordism. The morphism $\mathcal{K}_{\emptyset}(\bigcup):\mathcal{K}_{\emptyset}(T)\otimes_{\mathbf{k}} V \to \mathcal{K}_{\emptyset}(T)$ is given by $\mathcal{K}_{\emptyset}(\bigcup)(x\otimes v_{+})=0$ and $\mathcal{K}_{\emptyset}(\bigcup)(x\otimes v_{-})=x$. Thus, the persistent Khovanov homology of $\bigcup$ is
\begin{equation*}
    \im H^{\ast}(\bigcup;\mathcal{F})=H^{\ast}(T;\mathcal{F}).
\end{equation*}
Let $\raisebox{-0.15cm}{\saddle{0.3}} : T\to T'$ be the morphism of tangles with a local saddle cobordism. We have a morphism $\mathcal{K}_{\emptyset}(\raisebox{-0.15cm}{\saddle{0.3}}):\mathcal{K}_{\emptyset}(T)\to \mathcal{K}_{\emptyset}(T')$ of cochain complexes. Let $\widetilde{T}$ be the tangle obtained by changing $\raisebox{-0.15cm}{\udarc{0.3}}$ of $T$ into $\raisebox{-0.15cm}{\leftcross{0.2}}$. By \cite{bar2005khovanov}, one obtains a cochain complex $$\mathcal{K}_{\emptyset}(\widetilde{T}) = \mathcal{K}_{\emptyset}(T)[-1] \oplus \mathcal{K}_{\emptyset}(T')$$ with the differential given by $\widetilde{d}(z,z') = (-dz,\mathcal{K}_{\emptyset}(\raisebox{-0.15cm}{\saddle{0.3}})(z) + d'z')$, where $\mathcal{K}_{\emptyset}(T)[-1]$ is the height shift of $\mathcal{K}_{\emptyset}(T)$ given by $\mathcal{K}_{\emptyset}(T)[-1]^{p}=\mathcal{K}_{\emptyset}(T)^{p+1}$. Here, $z \in \mathcal{K}_{\emptyset}(T)[-1]$, $z' \in \mathcal{K}_{\emptyset}(T')$, and $d$, $d'$ are the differentials of $\mathcal{K}_{\emptyset}(T)[-1]$ and $\mathcal{K}_{\emptyset}(T')$, respectively. Thus, the morphism $\mathcal{K}_{\emptyset}(\raisebox{-0.15cm}{\saddle{0.3}}):\mathcal{K}_{\emptyset}(T)\to \mathcal{K}_{\emptyset}(T')$ is given by $\mathcal{K}_{\emptyset}(\raisebox{-0.15cm}{\saddle{0.3}})(z)=p_{1}\widetilde{d}z$. Here, $p_{1}:\mathcal{K}_{\emptyset}(\widetilde{T})\to \mathcal{K}_{\emptyset}(T')$ is the projection onto the component $\mathcal{K}_{\emptyset}(T')$. Therefore, one has a $\mathbf{k}$-module homomorphism $(p_{1}\widetilde{d})^{\ast}:H^{\ast}(T;\mathcal{F})\to H^{\ast}(T';\mathcal{F})$ given by $(p_{1}\widetilde{d})^{\ast}([z])=[p_{1}\widetilde{d}z]$ for any cohomology class $[z]\in H^{\ast}(T;\mathcal{F})$. It follows that the persistent Khovanov homology of the saddle morphism $\raisebox{-0.15cm}{\saddle{0.3}}$ is
\begin{equation*}
  \im H^{\ast}(\raisebox{-0.15cm}{\saddle{0.3}};\mathcal{F})=\im (p_{1}\widetilde{d})^{\ast}.
\end{equation*}
Besides, the morphisms of the Khovanov cochain complexes induced by the three Reidemeister moves are chain homotopy equivalences, and the corresponding morphisms of the Khovanov homology are isomorphisms. Therefore, for any persistence tangle with boundary $B$, the persistent Khovanov homology is a composition of a sequence of the three types of morphisms mentioned above and can be computed step by step.

\subsection{The construction of functors on tangles}
In \cite{khovanov2000categorification}, Khovanov provides a construction of a functor from the category of $(1,1)$-tangles to the category of modules. In \cite{khovanov2002functor}, he assigns graded bimodules to tangle smoothings by considering all closures of tangles. However, these constructions have limitations for our application to persistent homology. In this section, we will present a different construction of $\mathbf{k}$-modules on tangles.

\begin{wrapfigure}{r}{0.3\textwidth}
    \centering
    \includegraphics[width=0.3\textwidth]{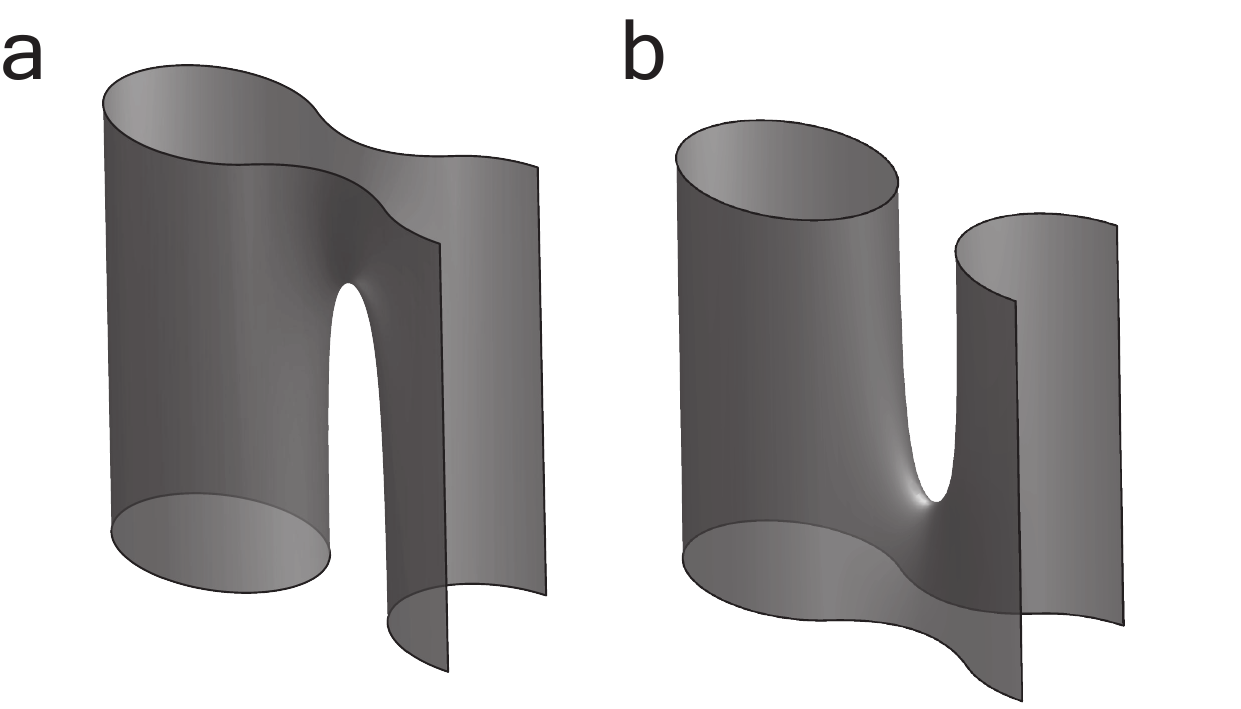} 
    \caption{The tangle cobordisms corresponding to the saddle maps in our construction.}
    \label{fig:arc}
\end{wrapfigure}
Let us define the functor $\mathcal{G}:\mathcal{C}ob^{3}(B)\to \mathcal{M}od_{\mathbf{k}}$ as follows. Let $V$ be the $\mathbf{k}$-module generated by the elements $v_{+}$ and $v_{-}$, and let $W$ be the $\mathbf{k}$-module generated by an element $w$.
For a tangle $T$ in $\mathcal{C}ob^{3}(B)$, the $\mathbf{k}$-module $\mathcal{G}(T)$ is the tensor product $\underbrace{W \otimes_{\mathbf{k}} \cdots \otimes_{\mathbf{k}} W}_{t(T)} \otimes \underbrace{V \otimes_{\mathbf{k}} \cdots \otimes_{\mathbf{k}} V}_{r(T)}$, where $r(T)$ and $t(T)$ represent the number of circles and arcs in $T$, respectively. Here, $V = \mathbf{k}\{v_{+}, v_{-}\}$ and $W = \mathbf{k}\{w\}$ are free $\mathbf{k}$-modules.

The functor $\mathcal{G}$ is defined as follows:
\begin{equation*}
  \begin{split}
      &  \mathcal{G}(\raisebox{-0.15cm}{\hsaddle{0.3}}:\raisebox{-0.15cm}{\clrarc{0.3}}\to \raisebox{-0.15cm}{\cudarc{0.3}}): W\otimes W \to W\otimes W, \quad w\otimes w \mapsto 0,\\
      &  \mathcal{G}(\raisebox{-0.15cm}{\saddle{0.3}}:\raisebox{-0.15cm}{\cudarcc{0.3}}\to \raisebox{-0.15cm}{\cudarco{0.3}}): W \to W\otimes V, \quad w \mapsto w\otimes v_{-},\\
      &  \mathcal{G}(\raisebox{-0.15cm}{\hsaddle{0.3}}:\raisebox{-0.15cm}{\cudarco{0.3}}\to \raisebox{-0.15cm}{\cudarcc{0.3}} ): W\otimes V \to W, \quad \left\{
    \begin{array}{ll}
      w \otimes v_{+} \mapsto w, \\
      w \otimes v_{-} \mapsto 0,
    \end{array}
  \right.
  \end{split}
\end{equation*}
and $\mathcal{G}$ coincides with $\mathcal{F}$ on the maps corresponding to the operations on the components of circles described in Section \ref{section:tqft}.
Here, the square boxes indicate that the arcs or circles within them are independent components in the tangle, in contrast to the arcs in the round boxes in $\raisebox{-0.15cm}{\udarc{0.3}}$ which represent local arcs within the tangle. Besides, in the above construction, the degree of $w$ is set to be $-1$.

\begin{proposition}
The construction $\mathcal{G}:\mathcal{C}ob^{3}(B)\to \mathcal{M}od_{\mathbf{k}}$ is functorial.
\end{proposition}
\begin{proof}
Recall that the construction $\mathcal{G}$ coincides with $\mathcal{F}$ on circles and the map between circles. We will focus on mappings that include arcs. To show $\mathcal{G}$ is a functor, the nontrivial steps are to verify the following diagrams commute.
\begin{equation*}
  \xymatrix{
  \text{I:}\quad
  \mathcal{G}(\raisebox{-0.15cm}{\cudarcoo{0.3}})\ar@{->}[rr]^-{\mathcal{G}(\raisebox{-0.15cm}{\hsaddle{0.3}})\otimes \mathrm{id}}\ar@{->}[d]_{\mathrm{id}\otimes m}&& \mathcal{G}(\raisebox{-0.15cm}{\cudarco{0.3}})\ar@{->}[d]^{\mathcal{G}(\raisebox{-0.15cm}{\hsaddle{0.3}})}\\
  \mathcal{G}(\raisebox{-0.15cm}{\cudarco{0.3}})\ar@{->}[rr]^-{\mathcal{G}(\raisebox{-0.15cm}{\hsaddle{0.3}})}& &\mathcal{G}(\raisebox{-0.15cm}{\cudarcc{0.3}})
  }\quad \text{II:}\quad   \xymatrix{
  \mathcal{G}(\raisebox{-0.15cm}{\cudarco{0.3}})\ar@{->}[rr]^-{\mathcal{G}(\raisebox{-0.15cm}{\hsaddle{0.3}})}\ar@{->}[d]_{\mathrm{id}\otimes \Delta}&& \mathcal{G}(\raisebox{-0.15cm}{\cudarcc{0.3}})\ar@{->}[d]^{\mathcal{G}(\raisebox{-0.15cm}{\saddle{0.3}})}\\
  \mathcal{G}(\raisebox{-0.15cm}{\cudarcoo{0.3}})\ar@{->}[rr]^-{\mathcal{G}(\raisebox{-0.15cm}{\hsaddle{0.3}})\otimes \mathrm{id}}& &\mathcal{G}(\raisebox{-0.15cm}{\cudarco{0.3}})
  }
\end{equation*}
Indeed, a step-by-step calculation shows that
\begin{equation*}
  \begin{array}{ccccc}
    \mathcal{G}(\raisebox{-0.15cm}{\cudarcoo{0.3}})& \stackrel{\mathcal{G}(\raisebox{-0.15cm}{\hsaddle{0.3}})\otimes \mathrm{id}}{\xrightarrow{\hspace*{1.5cm}}} & \mathcal{G}(\raisebox{-0.15cm}{\cudarco{0.3}})& \stackrel{\mathcal{G}(\raisebox{-0.15cm}{\hsaddle{0.3}})}{\xrightarrow{\hspace*{1.5cm}}} & \mathcal{G}(\raisebox{-0.15cm}{\cudarcc{0.3}})  \\
    w\otimes v_{+}\otimes v_{+} & \xmapsto{\hspace*{1cm} } &w\otimes v_{+} &  \xmapsto{\hspace*{1cm} } & w \\
    w\otimes v_{+}\otimes v_{-}  & \xmapsto{\hspace*{1cm} } & w\otimes v_{-} & \xmapsto{\hspace*{1cm} }& 0 \\
    w\otimes v_{-}\otimes v_{+}  & \xmapsto{\hspace*{1cm} } & 0 & \xmapsto{\hspace*{1cm} }& 0 \\
    w\otimes v_{-}\otimes v_{-}  & \xmapsto{\hspace*{1cm} } & 0 & \xmapsto{\hspace*{1cm} }& 0 \\
  \end{array}
\end{equation*}
and
\begin{equation*}
  \begin{array}{ccccc}
    \mathcal{G}(\raisebox{-0.15cm}{\cudarcoo{0.3}})& \stackrel{\mathrm{id}\otimes m}{\xrightarrow{\hspace*{1.5cm}}} & \mathcal{G}(\raisebox{-0.15cm}{\cudarco{0.3}})& \stackrel{\mathcal{G}(\raisebox{-0.15cm}{\hsaddle{0.3}})}{\xrightarrow{\hspace*{1.5cm}}} & \mathcal{G}(\raisebox{-0.15cm}{\cudarcc{0.3}})  \\
    w\otimes v_{+}\otimes v_{+} & \xmapsto{\hspace*{1cm} } &w\otimes v_{+} &  \xmapsto{\hspace*{1cm} } & w \\
    w\otimes v_{+}\otimes v_{-}  & \xmapsto{\hspace*{1cm} } & w\otimes v_{-} & \xmapsto{\hspace*{1cm} }& 0 \\
    w\otimes v_{-}\otimes v_{+}  & \xmapsto{\hspace*{1cm} } & w\otimes v_{-} & \xmapsto{\hspace*{1cm} }& 0 \\
    w\otimes v_{-}\otimes v_{-}  & \xmapsto{\hspace*{1cm} } & 0 & \xmapsto{\hspace*{1cm} }& 0. \\
  \end{array}
\end{equation*}
For the second diagram, we have that
\begin{equation*}
  \begin{array}{ccccc}
    \mathcal{G}(\raisebox{-0.15cm}{\cudarco{0.3}}) & \stackrel{\mathcal{G}(\raisebox{-0.15cm}{\hsaddle{0.3}})}{\xrightarrow{\hspace*{1.5cm}}} & \mathcal{G}(\raisebox{-0.15cm}{\cudarcc{0.3}}) & \stackrel{\mathcal{G}(\raisebox{-0.15cm}{\saddle{0.3}})}{\xrightarrow{\hspace*{1.5cm}}} & \mathcal{G}(\raisebox{-0.15cm}{\cudarco{0.3}}) \\
    w\otimes v_{+} & \xmapsto{\hspace*{1cm} } & w & \xmapsto{\hspace*{1cm} } & w\otimes v_{-} \\
    w\otimes v_{-} & \xmapsto{\hspace*{1cm} } & 0 & \xmapsto{\hspace*{1cm} } & 0
  \end{array}
\end{equation*}
and
\begin{equation*}
  \begin{array}{ccccc}
    \mathcal{G}(\raisebox{-0.15cm}{\cudarco{0.3}}) & \stackrel{\mathrm{id}\otimes \Delta}{\xrightarrow{\hspace*{1.5cm}}} &  \mathcal{G}(\raisebox{-0.15cm}{\cudarcoo{0.3}}) & \stackrel{\mathcal{G}(\raisebox{-0.15cm}{\hsaddle{0.3}})\otimes \mathrm{id}}{\xrightarrow{\hspace*{1.5cm}}} & \mathcal{G}(\raisebox{-0.15cm}{\cudarco{0.3}}) \\
    w\otimes v_{+} & \xmapsto{\hspace*{1cm} } & w\otimes v_{+}\otimes v_{-}+w\otimes v_{-}\otimes v_{+} & \xmapsto{\hspace*{1cm} } & w\otimes v_{-} \\
    w\otimes v_{-} & \xmapsto{\hspace*{1cm} } & w\otimes v_{-}\otimes v_{-} & \xmapsto{\hspace*{1cm} } & 0.
  \end{array}
\end{equation*}
The desired result follows. The remaining verifications are straightforward.
\end{proof}

Note that the relations $(S)$, $(T)$, and $(4Tu)$ occur at the components of cobordism between closed curves. Therefore, the functor $\mathcal{G}:\mathcal{C}ob^{3}(B) \to \mathcal{M}od_{\mathbf{k}}$ can descend to a functor $\mathcal{G}:\mathbf{k}\mathcal{C}ob^{3}_{/l}(B) \to \mathcal{M}od_{\mathbf{k}}$ from the additive category $\mathbf{k}\mathcal{C}ob^{3}_{/l}(B)$ to the abelian category $\mathcal{M}od_{\mathbf{k}}$. Thus we can obtain a functor $\mathcal{G}^{\bullet}:\mathbf{Ch}^{\bullet}(\mathrm{Mat}(\mathbf{k}\mathcal{C}ob^3_{/l}(B)))\to \mathbf{Ch}^{\bullet}(\mathcal{M}od_{\mathbf{k}})$ between the category of cochain complexes.
Now, we will give the detailed construction of the cochain complex of $\mathbf{k}$-module derived from $\mathcal{G}$. For a tangle $T$, let $\mathcal{G}[[T]]^{k}=\bigoplus\limits_{h(s)=k}\mathcal{G}(T_{s})$. And the map $\mathcal{G}(d)^{k}=\mathcal{G}(d^{k}):\mathcal{G}[[T]]^{k}\to \mathcal{G}[[T]]^{k+1}$ is given by $\mathcal{G}(d^{k})=\sum\limits_{\xi}(-1)^{\sgn(\xi)} \mathcal{G}(d_{\xi})$. Since $d^{k+1}\circ d^{k}=0$, we have $\mathcal{G}(d^{k+1})\circ\mathcal{G}(d^{k})=\mathcal{G}(d^{k+1}\circ d^{k})=0$. Hence, the construction $(\mathcal{G}[[T]]^{\ast},\mathcal{G}(d)^{\ast})$ is a cochain complex. Let $\mathcal{G}Kh^{p}(T)=\mathcal{G}[[T]]^{p -n_{-}}$ and $\mathcal{G}(d_{T})^{p}=\mathcal{G}(d)^{p-n_{-}}$. Thus, we have the following result.

\begin{proposition}
The construction $(\mathcal{G}Kh^{\ast}(T),\mathcal{G}(d_{T})^{\ast})$ is a cochain complex.
\end{proposition}

For any element $x$ in the cochain complex $\mathcal{G}Kh^{p}(T)$, we define the quantum grading of $x$ by $\Phi(x) = p + n_{+} - n_{-} + \theta(x)$, where $\theta(x)$ is obtained by taking $\theta(v_{+}) = 1$, $\theta(v_{-}) = -1$, and $\theta(w)=-1$.

\begin{lemma}[\cite{weibel1994introduction}]\label{lemma:functor}
Let $\mathcal{A}$ and $\mathcal{B}$ be additive categories. Any additive functor $F:\mathcal{A} \to \mathcal{B}$ induces an additive functor $F^{\bullet}:\mathbf{Ch}^{\bullet}(\mathcal{A}) \to \mathbf{Ch}^{\bullet}(\mathcal{B})$ that preserves homotopy equivalences.
\end{lemma}

Recall that we have the functor $Kh_{B}: \mathcal{C}ob^{4}(B) \to \mathbf{Ch}^{\bullet}(\mathrm{Mat}(\mathbf{k}\mathcal{C}ob^3_{/l}(B)))$. By composing it with $\mathcal{G}^{\bullet}: \mathbf{Ch}^{\bullet}(\mathrm{Mat}(\mathbf{k}\mathcal{C}ob^3_{/l}(B))) \to \mathbf{Ch}^{\bullet}(\mathcal{M}od_{\mathbf{k}})$, we obtain the functor $\mathcal{G}^{\bullet}Kh_{B}: \mathcal{C}ob^{4}(B) \to \mathbf{Ch}^{\bullet}(\mathcal{M}od_{\mathbf{k}})$, which maps the category of tangles with boundary $B$ to the category of cochain complexes of $\mathbf{k}$-modules.

\begin{theorem}\label{theorem:invariant}
The functor $\mathcal{G}^{\bullet}Kh_{B}:\mathcal{C}ob^{4}(B)\to \mathbf{Ch}^{\bullet}(\mathcal{M}od_{\mathbf{k}})$ maps isotopy classes of tangles to homotopy classes of cochain complexes.
\end{theorem}
\begin{proof}
Note that the functor $\mathcal{G}: \mathbf{k}\mathcal{C}ob^3_{/l}(B) \to \mathcal{M}od_{\mathbf{k}}$ is additive, and it extends to an additive functor $\mathrm{Mat}(\mathbf{k}\mathcal{C}ob^3_{/l}(B)) \to \mathcal{M}od_{\mathbf{k}}$. The desired result follows from a variant of \cite[Theorem 4]{bar2005khovanov} and Lemma \ref{lemma:functor}.
\end{proof}

\section{Planar tangles and persistent Khovanov homology}

In the study of persistent Khovanov homology for tangles, it is often the case that the boundary of the tangle does not remain fixed as the tangle evolves over persistence parameter. This presents a challenge for the application of persistence tangles. A natural approach is to consider that as the persistence parameter increases, the tangle at earlier times can be viewed as an interior part of the tangle at later times. The relationship between these two tangles can be described using operations induced by input planar tangles.

\subsection{Input planar tangle}

A \emph{$d$-input planar tangle} consists of a large output disk equipped with $d$ input disks, along with a collection of disjoint embedded arcs that are either closed or have endpoints on the boundary. These input disks are sequentially numbered from 1 to $d$, and both the input disks and the output disk are marked with $\ast$ as base points.

Let $\mathcal{T}(k)$ be the collection of all the classes of tangles with $k$ endpoints up to Reidemeister moves. Suppose $D$ is a $d$-input planar tangle such that there are $k_{r}$ endpoints of arcs on the $r$-th input disk in $D$ for $R=1,2,\dots,d$. Then one has an operation
\begin{equation*}
  D:\mathcal{T}(k_{1})\times \cdots \times \mathcal{T}(k_{d}) \to \mathcal{T}(k),
\end{equation*}
which embeds $d$ tangles, each with $k_{1}, \dots, k_{d}$ endpoints respectively, into the $d$-input planar tangle $D$ by connecting their endpoints, resulting in a new tangle. Let $\mathcal{P}(k)$ be the vector space generated by the elements in $\mathcal{T}(k)$. Then the collection $\{\mathcal{P}(k)\}_{k \geq 0}$, equipped with the operation $D$, forms a planar algebra. For more details on planar algebras, refer to \cite{jones2021planar}.

Now, let $D$ be a 1-input planar tangle. We can obtain an operation
\begin{equation*}
  D:\mathcal{T}(k_{1}) \to \mathcal{T}(k)
\end{equation*}
by embedding a tangle $T$ into $D$, resulting in a larger tangle $D(T)$, with $T$ as a part of $D(T)$, as shown in Figure \ref{figure:input}.
\begin{figure}[htb!]
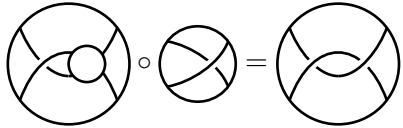

  \centering
  $\Tanglo[1.6]\circ \Tanglx[1]= \Tangle[1.6]$\\
  \caption{An example of the operation of a 1-input planar tangle.}\label{figure:input}
\end{figure}

Our goal in this work is to establish the distance-based persistent Khovanov homology of tangles. A natural idea is to determine whether we can obtain a morphism $Kh(\mathcal{T}(k_{1})) \to Kh(\mathcal{T}(k))$ of cochain complexes. Unfortunately, it is challenging to construct such a morphism of cochain complexes. Even with the constructions from TQFT, we have not been able to establish a morphism $\mathcal{F}^{\bullet}Kh(\mathcal{T}(k_{1})) \to \mathcal{F}^{\bullet}Kh(\mathcal{T}(k))$ of cochain complexes.

\subsection{The category $\mathcal{P}la$}

Consider the category $\mathcal{P}la$ of tangles, where the objects are tangles and the morphisms are given by maps $T \to T' = D(T)$ for some 1-input planar tangle $D$. In this setting, any morphism in $\mathcal{P}la$ can be viewed as an inclusion of 1-dimensional manifolds, where arcs are mapped to either arcs or circles, and circles are mapped to circles.

For any morphism $T \to T'$, we can associate cochain complexes $(\mathcal{G}Kh^{\ast}(T), \mathcal{G}(d_{T})^{\ast})$ and $(\mathcal{G}Kh^{\ast}(T'), \mathcal{G}(d_{T'})^{\ast})$. Our goal is to construct a map $\Psi: (\mathcal{G}Kh^{\ast}(T), \mathcal{G}(d_{T})^{\ast}) \to (\mathcal{G}Kh^{\ast}(T'), \mathcal{G}(d_{T'})^{\ast})$.

One essential difficulty arises when $T'$ has more connected components than $T$, in which case the induced map on chain complexes cannot be defined. For instance, if we embed a single arc into a tangle consisting of two arcs, it is unclear how to interpret the corresponding map $W \to W \otimes W$. To avoid such situations, we require that in the morphism $T \to T'$, the target $T'$ does not contain more connected components or crossings than the source $T$. Under this restriction, we obtain a subcategory of $\mathcal{P}la$, which we denote by $\mathcal{P}la^{0}$.

In the category $\mathcal{P}la^{0}$, the morphisms $T \to T'$ are tangles with the same number of crossings, such that each crossing of $T$ corresponds uniquely to a crossing of $T'$. Consequently, the induced morphism on chain complexes is determined at each state $s$ by $T_s \longrightarrow T'_s$. Therefore, the restriction of the map $\Psi$ to each state is defined on arcs or circles.

The map $\Psi$ is a $\mathbf{k}$-module homomorphism defined as follows:
\begin{equation*}
  \Psi: \mathcal{G}(\raisebox{-0.15cm}{\cudarcc{0.3}}) \to \mathcal{G}(\raisebox{-0.15cm}{\cscircle{0.3}}), \quad w \mapsto v_{-},
\end{equation*}
and $\Psi$ acts as the identity on the identity maps $\raisebox{-0.15cm}{\cscircle{0.3}}\to\raisebox{-0.15cm}{\cscircle{0.3}}$ and $\raisebox{-0.15cm}{\cudarcc{0.3}}\to\raisebox{-0.15cm}{\cudarcc{0.3}}$ of independent components. In other words, $\Psi$ maps arcs to arcs and circles to circles wherever the structure of the tangle is preserved, and performs the specified homomorphism on components that transition between arcs and circles.

\begin{theorem}
The map $\Psi:\mathcal{G}Kh^{\ast}(T)\to \mathcal{G}Kh^{\ast}(T')$ is a morphism of cochain complexes.
\end{theorem}
\begin{proof}
To prove $\Psi$ is a morphism of cochain complexes, it suffices to show $\mathcal{G}(d_{\xi})\circ \Psi=\Psi\circ \mathcal{G}(d_{\xi})$. Here, $d_{\xi}:T_{s}\to T_{s'}$ is a saddle map given by the edge
$\xi = (\xi_1, \dots, \xi_{i-1}, \star, \xi_{i+1}, \dots, \xi_{|\mathcal{X}(T)|}) \in \{0, 1, \star\}^{|\mathcal{X}(T)|}$ in the state cube that connect a state $s$ with a neighboring state $s'$ that differs by one position. Here, $\xi_{j} \in \{0, 1\}$ for $j\neq i$ and $\star$ indicates an edge connecting $0$ to $1$. Hence, we need to prove that the following four diagrams are commutative.
\begin{equation*}
\text{I:}\quad
  \xymatrix{
  \mathcal{G}(\raisebox{-0.15cm}{\clrarc{0.3}})\ar@{->}[rr]^{\mathcal{G}(\raisebox{-0.15cm}{\hsaddle{0.3}})}\ar@{->}[d]^{\Psi}&& \mathcal{G}(\raisebox{-0.15cm}{\cudarc{0.3}})\ar@{->}[d]^{\Psi}\\
  \mathcal{G}(\raisebox{-0.15cm}{\cudarco{0.3}})\ar@{->}[rr]^{\mathcal{G}(\raisebox{-0.15cm}{\hsaddle{0.3}})}& &\mathcal{G}(\raisebox{-0.15cm}{\cudarcc{0.3}})
  }\quad
  \text{II:}\quad
  \xymatrix{
  \mathcal{G}(\raisebox{-0.15cm}{\cudarc{0.3}})\ar@{->}[rr]^{\mathcal{G}(\raisebox{-0.15cm}{\saddle{0.3}})}\ar@{->}[d]^{\Psi}&& \mathcal{G}(\raisebox{-0.15cm}{\clrarc{0.3}})\ar@{->}[d]^{\Psi}\\
  \mathcal{G}(\raisebox{-0.15cm}{\cudarcc{0.3}})\ar@{->}[rr]^{\mathcal{G}(\raisebox{-0.15cm}{\saddle{0.3}})}& &\mathcal{G}(\raisebox{-0.15cm}{\cudarco{0.3}})
  }
\end{equation*}
\begin{equation*}
\text{III:}\quad
  \xymatrix{
  \mathcal{G}(\raisebox{-0.15cm}{\cudarco{0.3}})\ar@{->}[rr]^{\mathcal{G}(\raisebox{-0.15cm}{\hsaddle{0.3}})}\ar@{->}[d]^{\Psi}&& \mathcal{G}(\raisebox{-0.15cm}{\cudarcc{0.3}})\ar@{->}[d]^{\Psi}\\
  \mathcal{G}(\raisebox{-0.15cm}{\ccircle{0.3}})\ar@{->}[rr]^{m}& &\mathcal{G}(\raisebox{-0.15cm}{\cscircle{0.3}})
  }\quad
  \text{IV:}\quad
  \xymatrix{
  \mathcal{G}(\raisebox{-0.15cm}{\cudarcc{0.3}})\ar@{->}[rr]^{\mathcal{G}(\raisebox{-0.15cm}{\saddle{0.3}})}\ar@{->}[d]^{\Psi}&& \mathcal{G}(\raisebox{-0.15cm}{\cudarco{0.3}})\ar@{->}[d]^{\Psi}\\
  \mathcal{G}(\raisebox{-0.15cm}{\cscircle{0.3}})\ar@{->}[rr]^{\Delta}& &\mathcal{G}(\raisebox{-0.15cm}{\ccircle{0.3}})
  }
\end{equation*}
We will only verify the third diagram. A straightforward calculation shows that
\begin{equation*}
  \begin{array}{ccccc}
    \mathcal{G}(\raisebox{-0.15cm}{\cudarco{0.3}}) & \stackrel{\mathcal{G}(\raisebox{-0.15cm}{\hsaddle{0.3}})}{\xrightarrow{\hspace*{1.5cm}}} & \mathcal{G}(\raisebox{-0.15cm}{\cudarcc{0.3}}) & \stackrel{\Psi}{\xrightarrow{\hspace*{1.5cm}}} & \mathcal{G}(\raisebox{-0.15cm}{\cscircle{0.3}}) \\
    w\otimes v_{+} & \xmapsto{\hspace*{1cm} } & w & \xmapsto{\hspace*{1cm} } & v_{-} \\
    w\otimes v_{-} & \xmapsto{\hspace*{1cm} } & 0 & \xmapsto{\hspace*{1cm} } & 0
  \end{array}
\end{equation*}
and
\begin{equation*}
  \begin{array}{ccccc}
    \mathcal{G}(\raisebox{-0.15cm}{\cudarco{0.3}}) & \stackrel{\Psi}{\xrightarrow{\hspace*{1.5cm}}} & \mathcal{G}(\raisebox{-0.15cm}{\ccircle{0.3}}) & \stackrel{m}{\xrightarrow{\hspace*{1.5cm}}} & \mathcal{G}(\raisebox{-0.15cm}{\cscircle{0.3}}) \\
    w\otimes v_{+} & \xmapsto{\hspace*{1cm} } & v_{-}\otimes v_{+} & \xmapsto{\hspace*{1cm} } & v_{-} \\
    w\otimes v_{-} & \xmapsto{\hspace*{1cm} } & v_{-}\otimes v_{-} & \xmapsto{\hspace*{1cm} } & 0.
  \end{array}
\end{equation*}
Thus, Diagram III commutes. The commutativity of the other diagrams can be verified similarly, following analogous calculations.
\end{proof}

\begin{example}
Now, we will give an example of tangles with more independent components. Consider the following diagram.
\begin{equation*}
  \xymatrix{
  \mathcal{G}(\raisebox{-0.15cm}{\caoa{0.3}})\ar@{->}[r]^{\Psi}\ar@{->}[d]_{\mathrm{id}\otimes\mathcal{G}(\raisebox{-0.15cm}{\hsaddle{0.3}})}&\mathcal{G}(\raisebox{-0.15cm}{\cudarcoo{0.3}})
  \ar@{->}[r]^{\Psi}\ar@{->}[d]^{\mathcal{G}(\raisebox{-0.15cm}{\hsaddle{0.3}})\otimes \mathrm{id}}&\mathcal{G}(\raisebox{-0.15cm}{\cooo{0.3}})\ar@{->}[d]^{m\otimes \mathrm{id}}\\
  \mathcal{G}(\raisebox{-0.15cm}{\caarc{0.3}})\ar@{->}[r]^{\Psi}&\mathcal{G}(\raisebox{-0.15cm}{\cudarco{0.3}})\ar@{->}[r]^{\Psi}&\mathcal{G}(\raisebox{-0.15cm}{\ccircle{0.3}})
  }
\end{equation*}
It is worth noting that $\mathcal{G}(\raisebox{-0.15cm}{\hsaddle{0.3}})\otimes \mathrm{id}=\mathrm{id}\otimes m$ and $m\otimes \mathrm{id}=\mathrm{id}\otimes m$. The corresponding element mappings and their associated diagrams are listed as follows.
\begin{equation*}
  \xymatrix{
 w\otimes v_{+}\otimes w\ar@{|->}[r]\ar@{|->}[d]&v_{-}\otimes v_{+}\otimes w\ar@{|->}[d]\ar@{|->}[r]&v_{-}\otimes v_{+}\otimes v_{-}\ar@{|->}[d]\\
 w\otimes w\ar@{|->}[r]&v_{-}\otimes w\ar@{|->}[r]&v_{-}\otimes v_{-}
  }
  \end{equation*}
  \begin{equation*}
  \xymatrix{
 w\otimes v_{-}\otimes w\ar@{|->}[r]\ar@{|->}[d]& w\otimes v_{-}\otimes v_{-}\ar@{->}[d]\ar@{->}[r]&v_{-}\otimes v_{-}\otimes v_{-}\ar@{|->}[d]\\
 0\ar@{|->}[r]&0\ar@{|->}[r]&0
  }
\end{equation*}
The calculation shows that the above diagram of $\mathbf{k}$-modules is commutative.
\end{example}

\begin{theorem}
The construction $\mathcal{G}^{\bullet}Kh:\mathcal{P}la^{0}\to \mathbf{Ch}^{\bullet}(\mathcal{M}od_{\mathbf{k}})$ is functorial.
\end{theorem}
\begin{proof}
Let $T\stackrel{D}{\to} T'\stackrel{D'}{\to} T''$ be morphisms of tangles in the category $\mathcal{P}la^{0}$. We need to prove that the following diagram commutes:
\begin{equation*}
  \xymatrix{
  \mathcal{G}^{\bullet}Kh(T)\ar@{->}[rr]^{\Psi_{D}}\ar@{->}[rd]_{\Psi_{D'\circ D}}&&\mathcal{G}^{\bullet}Kh(T')\ar@{->}[ld]^{\Psi_{D'}}\\
  &\mathcal{G}^{\bullet}Kh(T'').&
  }
\end{equation*}
Here, $D'\circ D$ is the composition of morphisms in the category $\mathcal{P}la^{0}$. In other words, we need to prove that $\Psi_{D'\circ D}=\Psi_{D'}\circ \Psi_{D}$. It suffices to prove the diagram
\begin{equation*}
  \xymatrix{
  \mathcal{G}(T)\ar@{->}[rr]^{\mathcal{G}(T\to T')}\ar@{->}[rd]_{\mathcal{G}(T\to T'')}&&\mathcal{G}(T')\ar@{->}[ld]^{\Psi_{\mathcal{G}(T'\to T'')}}\\
  &\mathcal{G}(T'')&
  }
\end{equation*}
is commutative. We only need to verify the commutativity for the two cases of morphisms $T \to T' \to T''$:
\begin{equation*}
  \raisebox{-0.15cm}{\cudarcc{0.3}} \to \raisebox{-0.15cm}{\cscircle{0.3}} \to \raisebox{-0.15cm}{\cscircle{0.3}},\quad \raisebox{-0.15cm}{\cudarcc{0.3}} \to \raisebox{-0.15cm}{\cudarcc{0.3}} \to \raisebox{-0.15cm}{\cscircle{0.3}}.
\end{equation*}
This follows from a straightforward step-by-step computation. The remaining part of the proof can be checked directly.
\end{proof}

It is worth noting that, at present, there is no definition of isotopy for tangles with different boundaries. Consequently, we do not have a result stating that the functor
$\mathcal{G}^{\bullet}Kh:\mathcal{P}la^{0}\to \mathbf{Ch}^{\bullet}(\mathcal{M}od_{\mathbf{k}})$
maps isotopy classes of tangles to homotopy classes of cochain complexes.

\subsection{Homology functors for tangles}

In the previous sections, we introduced a new construction $\mathcal{G}:\mathcal{C}ob^{3}(B)\to\mathcal{M}od_{\mathbf{k}}$ for tangles. This construction is functorial and leads to two functors: $\mathcal{G}^{\bullet}Kh_{B} : \mathcal{C}ob^{4}(B) \to \mathbf{Ch}^{\bullet}(\mathcal{M}od_{\mathbf{k}})$ and $\mathcal{G}^{\bullet}Kh:\mathcal{P}la^{0}\to\mathbf{Ch}^{\bullet}(\mathcal{M}od_{\mathbf{k}})$. The functor $\mathcal{G}^{\bullet}Kh_{B}$ is a tangle invariant up to homotopy equivalence, but it has limitations for applications because it requires the boundaries of tangles to be fixed. In contrast, although functor $\mathcal{G}^{\bullet}Kh$ does not capture tangle invariants, it  has greater potential for application.

For a given tangle $T$, the constructions $\mathcal{G}^{\bullet}Kh_{\partial T}(T)$ and $\mathcal{G}^{\bullet}Kh(T)$ produce the same cochain complex. Thus, although $\mathcal{G}^{\bullet}Kh_{\partial T}$ and $\mathcal{G}^{\bullet}Kh$ are different functors, this does not impact the computation of the Khovanov homology of tangles. For practical purposes, we will use the homology functor associated with $\mathcal{G}^{\bullet}Kh$.

\begin{definition}
Let $T$ be a tangle. The \emph{Khovanov homology of $T$ associated with $\mathcal{G}$} is defined by
\begin{equation*}
  H^{p}(T;\mathcal{G})=H^{p}(\mathcal{G}^{\bullet}Kh(T)),\quad p\in \mathbb{Z}.
\end{equation*}
\end{definition}
The Khovanov homology associated with $\mathcal{G}$ is a functor $H^{p}(-;\mathcal{G}):\mathcal{P}la^{0}\to \mathcal{M}od_{\mathbf{k}}$. Moreover, if $\partial T=\emptyset$, the Khovanov homology associated with $\mathcal{G}$ reduces to the Khovanov homology of links, that is, $H^{p}(T;\mathcal{G})=H^{p}(T;\mathcal{F})$ for any $p$.
The Khovanov homology of tangles associated with $\mathcal{G}$ can be explicitly computed. The following computation provides a detailed example.

\begin{example}
Consider the tangle $T=\raisebox{-0.4cm}{\lcrossingarc{0.2}}$. The corresponding cochain complex $Kh(T)$ of $T$ in $\mathbf{Ch}^{\bullet}(\mathrm{Mat}(\mathbf{k}\mathcal{C}ob^3_{/l}(\partial T))$ is described as follows:
\begin{equation*}
  \xymatrix{
  0\ar@{->}[r]&\underset{-1}{\raisebox{-0.15cm}{\cudarcc{0.3}}}\ar@{->}[r]^{\raisebox{-0.15cm}{\saddle{0.3}}}& \underset{0}{\raisebox{-0.15cm}{\cudarco{0.3}}}\ar@{->}[r]&0.
  }
\end{equation*}
The cochain complex $Kh^{\ast}(T)$ collapses at heights $-1$ and $0$. The only nontrivial differential is $d^{-1}=\raisebox{-0.15cm}{\saddle{0.3}}:Kh^{-1}(T)\to Kh^{0}(T)$. Applying to the functor $\mathcal{G}$, we have a cochain complex of $\mathbf{k}$-modules
\begin{equation*}
  \xymatrix{
  0\ar@{->}[r]&W\ar@{->}[r]^-{d}& W\otimes V\ar@{->}[r]&0.
  }
\end{equation*}
Here, $dw=w\otimes v_{-}$. A straightforward calculation shows that
\begin{equation*}
  H^{p}(T;\mathcal{G})=\left\{
                         \begin{array}{ll}
                           \mathbf{k}\{[w\otimes v_{+}]\}, & \hbox{$p=0$;} \\
                           0, & \hbox{otherwise.}
                         \end{array}
                       \right.
\end{equation*}
Recall that $\deg w=-1$. Then the quantum grading of $w\otimes v_{+}$ is given by $-1$. Now, consider the tangle $T'=\raisebox{-0.4cm}{\rcrossingarc{0.2}}$. Then the cochain complex $Kh(T')$ is described as follows:
\begin{equation*}
  \xymatrix{
  0\ar@{->}[r]& \underset{0}{\raisebox{-0.15cm}{\cudarco{0.3}}}\ar@{->}[r]^{\raisebox{-0.15cm}{\hsaddle{0.3}}}&\underset{1}{\raisebox{-0.15cm}{\cudarcc{0.3}}}\ar@{->}[r]&0.
  }
\end{equation*}
The differential at dimension $0$ is given by $d^{0}=\raisebox{-0.15cm}{\saddle{0.3}}:Kh^{0}(T)\to Kh^{1}(T)$. Thus, we have a cochain complex of $\mathbf{k}$-modules
\begin{equation*}
  \xymatrix{
  0\ar@{->}[r]&W\otimes V\ar@{->}[r]^-{d}& W\ar@{->}[r]&0,
  }
\end{equation*}
where $d(w\otimes v_{+})=w$ and $d(w\otimes v_{-})=0$. The corresponding Khovanov homology is
\begin{equation*}
  H^{p}(T';\mathcal{G})=\left\{
                         \begin{array}{ll}
                           \mathbf{k}\{[w\otimes v_{-}]\}, & \hbox{$p=0$;} \\
                           0, & \hbox{otherwise.}
                         \end{array}
                       \right.
\end{equation*}
The quantum grading of $w\otimes v_{-}$ is $-1$. Now, consider the tangle $T''$ consisting of a single arc. It is clear that the Khovanov homology is
\begin{equation*}
  H^{p}(T'';\mathcal{G})=\left\{
                         \begin{array}{ll}
                           \mathbf{k}\{[w]\}, & \hbox{$p=0$;} \\
                           0, & \hbox{otherwise.}
                         \end{array}
                       \right.
\end{equation*}
The quantum grading of $w$ here is also $-1$. In this example, $T$, $T'$, and $T''$ are equivalent up to Reidemeister moves. Their corresponding Khovanov homology groups are also identical, with even the quantum gradings of the homology generators being equal.
\end{example}

\subsection{Application}

In Section \ref{section:persistence_tangle}, we defined persistent Khovanov homology of tangles within the category of tangles with fixed boundaries. However, in practical applications, it is uncommon to encounter the filtration of tangles with fixed boundaries. In this section, we present an application that describes how, with a given tangle in a metric space, one can construct persistent tangles within the category $\mathcal{P}la^{0}$, thereby obtaining the persistent Khovanov homology of tangles.

Let $(X,\leq)$ be a poset. Then $(X, \leq)$ can be regarded as a category, where the objects are the elements of $X$, and the morphisms are the pairs $(x, x')$ such that $x \leq x'$ for $x, x' \in X$.

\begin{definition}
A \emph{persistence tangle in category $\mathcal{P}la^{0}$} is a functor $\mathcal{P}:(X,\leq)\to \mathcal{P}la^{0}$.
\end{definition}
\begin{definition}
Let $\mathcal{P}:(X,\leq)\to \mathcal{P}la^{0}$ be a persistence tangle. The \emph{persistent Khovanov homology of tangles} is the composition of functors
\begin{equation*}
  \xymatrix{
  (X,\leq)\ar@{->}[r]^-{\mathcal{P}}&\mathcal{P}la^{0}\ar@{->}[r]^{H(-;\mathcal{G})}& \mathcal{M}od_{\mathbf{k}}.
  }
\end{equation*}
\end{definition}
For any $a \leq b$ in $X$, the $(a,b)$-persistent Khovanov  homology of   tangles $\mathcal{P}:(X,\leq)\to \mathcal{P}la^{0}$ is given by
\begin{equation*}
  H_{a,b}^{p}(\mathcal{P};\mathcal{G}) = \im \left(H^{p}(\mathcal{P}(a);\mathcal{G}) \to H^{p}(\mathcal{P}(b);\mathcal{G})\right), \quad p \in \mathbb{Z}.
\end{equation*}

\begin{definition}
Let \(\mathcal{C}\) be a category. A construction $F:(\mathbb{R},\le)\longrightarrow \mathcal{C}$ is called a \emph{piecewise persistence object} in $\mathcal{C}$ if there exists a finite partition
\[
-\infty = t_0 < t_1 < \cdots < t_m = +\infty
\]
such that, for each interval \(I_j = [t_{j-1},t_j)\), the restriction $F|_{I_j} : (I_j,\le) \longrightarrow \mathcal{C}$ is a functor.
\end{definition}

\begin{remark}
A piecewise persistence object can also be considered a persistence object. Let $F: (\mathbb{R}, \le) \longrightarrow \mathcal{C}$ be a piecewise persistence object with a finite partition $\varpi$ given by $\mathbb{R} = \bigsqcup_{j} I_{j}$. We regard $\mathbb{R}$ as a poset, denoted by $X(\mathbb{R}, \varpi)$, where the partial order $a \leq b$ holds if and only if $a, b \in I_j$ for some $j$. In this way, the construction $F_{\varpi}: X(\mathbb{R}, \varpi) \to \mathcal{C}$ is a persistence object.
\end{remark}

\begin{example}
Consider a tangle $T$ in a Euclidean plane. Suppose the arcs or circles of $T$ are of finite length. Fix a point $P$ as the center, and let $D_{\varepsilon}$ denote a  disk centered at $P$ with radius $\varepsilon$. For each $\varepsilon$, define the tangle $T_{\varepsilon} = T \cap D_{\varepsilon}$, which may be empty. Then the construction $\mathcal{P}: (\mathbb{R}, \leq) \to \mathcal{P}la^{0}$ defined by $\mathcal{P}(\varepsilon) = T_{\varepsilon}$ is a piecewise persistence tangle.
For any real numbers $a\leq b$ in a common partition interval, we have the corresponding $(a,b)$-persistent Khovanov homology of tangles $H_{a,b}^{\ast}(\mathcal{P};\mathcal{G})$.
\end{example}

\begin{example}
Let $C$ be a finite collection of compact curves in 3-dimensional Euclidean space, and let $q:C \to \mathbb{R}^{2}$ be a projection such that there are finitely many crossings, each of which is required to be a double point. Let $\{D_{\varepsilon}\}_{\varepsilon\in \mathbb{R}}$ be a family of disks in $\mathbb{R}^{2}$ with the same center. Then the intersection $T_{\varepsilon} = q(C) \cap D_{\varepsilon}$ is a tangle (or the empty set) for any $\varepsilon > 0$. This defines a piecewise persistence tangle $T_{\varepsilon}:(\mathbb{R}, \leq) \to \mathcal{P}la^{0}$, which can be used to compute the persistent Khovanov homology of tangles and extract topological features.
\end{example}

In practical applications, persistent tangles can be derived from one-dimensional manifolds embedded in three-dimensional space, or even from collections of non-smooth curves. By computing the persistent Khovanov homology of tangles, one can extract multi-scale topological features, which can then be used to analyze curve-type data. This highlights the significant potential of persistent Khovanov homology of tangles across various application domains in data science.

\section{Acknowledgments}

This work was supported in part by NIH grants R01GM126189, R01AI164266, and R35GM148196, National Science Foundation grants DMS2052983, DMS-1761320, and IIS-1900473, NASA  grant 80NSSC21M0023,   Michigan State University Research Foundation, and  Bristol-Myers Squibb  65109.
Jian was also supported by Natural Science Foundation of China (NSFC Grant No. 12401080) and Scientific Research Foundation of Chongqing University of Technology.
The authors would also like to thank Iker Torres for raising a question on an earlier version of this paper, which helped us to improve the presentation.

\bibliographystyle{plain}  
\bibliography{Reference}

\begin{thebibliography}{10}

\bibitem{baez2003higher}
John~C Baez and Laurel Langford.
\newblock Higher-dimensional algebra iv: 2-tangles.
\newblock {\em Advances in Mathematics}, 180(2):705--764, 2003.

\bibitem{bar2002khovanov}
Dror Bar-Natan.
\newblock On khovanov's categorification of the {J}ones polynomial.
\newblock {\em Algebraic \& Geometric Topology}, 2(1):337--370, 2002.

\bibitem{bar2005khovanov}
Dror Bar-Natan.
\newblock Khovanov's homology for tangles and cobordisms.
\newblock {\em Geometry \& Topology}, 9(3):1443--1499, 2005.

\bibitem{carlsson2009topology}
Gunnar Carlsson.
\newblock Topology and data.
\newblock {\em Bulletin of the American Mathematical Society}, 46(2):255--308,
  2009.

\bibitem{carter1993reidemeister}
J~Scott Carter and Masahico Saito.
\newblock Reidemeister moves for surface isotopies and their interpretation as
  moves to movies.
\newblock {\em Journal of Knot Theory and its Ramifications}, 2(03):251--284,
  1993.

\bibitem{chen2022persistent}
Jiahui Chen, Yuchi Qiu, Rui Wang, and Guo-Wei Wei.
\newblock Persistent laplacian projected {Omicron BA.4 and BA.5} to become new
  dominating variants.
\newblock {\em Computers in Biology and Medicine}, 151:106262, 2022.

\bibitem{chen2021evolutionary}
Jiahui Chen, Rundong Zhao, Yiying Tong, and Guo-Wei Wei.
\newblock Evolutionary de {Rham-Hodge} method.
\newblock {\em Discrete and continuous dynamical systems. Series B},
  26(7):3785, 2021.

\bibitem{edelsbrunner2002topological}
Edelsbrunner, Letscher, and Zomorodian.
\newblock Topological persistence and simplification.
\newblock {\em Discrete \& computational geometry}, 28:511--533, 2002.

\bibitem{fischer19942}
John~E Fischer~Jr.
\newblock $2 $-categories and $2 $-knots.
\newblock {\em Duke Math. J.}, 76(1):493--526, 1994.

\bibitem{jones2021planar}
Vaughan Jones.
\newblock Planar algebras.
\newblock {\em New Zealand Journal of Mathematics}, 52:1--107, 2021.

\bibitem{khovanov2000categorification}
Mikhail Khovanov.
\newblock A categorification of the {J}ones polynomial.
\newblock {\em Duke Mathematical Journal}, 101(3):359 -- 426, 2000.

\bibitem{khovanov2002functor}
Mikhail Khovanov.
\newblock A functor-valued invariant of tangles.
\newblock {\em Algebraic \& Geometric Topology}, 2(2):665--741, 2002.

\bibitem{langford19972}
Laurel Tamara~Fearnley Langford.
\newblock {\em 2-tangles as a free braided monoidal 2-category with duals}.
\newblock University of California, Riverside, 1997.

\bibitem{le1995representation}
Tu~Quoc~Thang Le and Jun Murakami.
\newblock Representation of the category of tangles by {Kontsevich's} iterated
  integral.
\newblock {\em Communications in mathematical physics}, 168:535--562, 1995.

\bibitem{roseman1998reidemeister}
Dennis Roseman.
\newblock Reidemeister-type moves for surfaces in four-dimensional space.
\newblock {\em Banach Center Publications}, 42(1):347--380, 1998.

\bibitem{shen2024knot}
Li~Shen, Hongsong Feng, Fengling Li, Fengchun Lei, Jie Wu, and Guo-Wei Wei.
\newblock Knot data analysis using multiscale {Gauss} link integral.
\newblock {\em Proceedings of the National Academy of Sciences}, accepted,
  2024.

\bibitem{shen2024evolutionary}
Li~Shen, Jian Liu, and Guo-Wei Wei.
\newblock Evolutionary {Khovanov} homology.
\newblock {\em AIMS Mathematics}, 9(9):26139--26165, 2024.

\bibitem{wang2020persistent}
Rui Wang, Duc~Duy Nguyen, and Guo-Wei Wei.
\newblock Persistent spectral graph.
\newblock {\em International journal for numerical methods in biomedical
  engineering}, 36(9):e3376, 2020.

\bibitem{weibel1994introduction}
Charles~A Weibel.
\newblock {\em An introduction to homological algebra}.
\newblock Number~38. Cambridge university press, 1994.

\end{thebibliography}

\end{document}